\documentclass[9pt,a4paper]{amsart}
\usepackage{amscd,amssymb,amsopn,amsmath,amsthm,graphics,amsfonts,enumerate,verbatim,calc}
\usepackage[dvips]{graphicx}

\usepackage{mathpazo}
\usepackage{color}
\usepackage{latexsym}
\usepackage{amsthm,amsfonts,amssymb,mathrsfs}
\usepackage{rotating}
\usepackage[leqno]{amsmath}
\usepackage{xspace}
\usepackage[all]{xy}
\usepackage{longtable}
\headsep=0.3cm \numberwithin{equation}{section}
\newtheorem{theorem}{Theorem}[section]
\newtheorem{lemma}[theorem]{Lemma}

\newtheorem{proposition}[theorem]{Proposition}
\newtheorem{corollary}[theorem]{Corollary}

\theoremstyle{definition}

\theoremstyle{remark}
\newtheorem{remark}[theorem]{Remark}
\newtheorem{fact}[theorem]{Fact}
\newtheorem{example}[theorem]{Example}
\newtheorem{observation}[theorem]{Observation}
\newtheorem{discussion}[theorem]{Discussion}
\newtheorem{question}[theorem]{Question}
\newtheorem{conjecture}[theorem]{Conjecture}

\newcommand{\sing}{\operatorname{sing}}

\newcommand{\Ass}{\operatorname{Ass}}

\newcommand{\grade}{\operatorname{grade}}

\newcommand{\Assh}{\operatorname{Assh}}
\newcommand{\Spec}{\operatorname{Spec}}

\newcommand{\rad}{\operatorname{rad}}
\newcommand{\cd}{\operatorname{cd}}

\newcommand{\Ht}{\operatorname{ht}}
\newcommand{\pd}{\operatorname{p.dim}}

\newcommand{\Proj}{\operatorname{Proj}}

\newcommand{\F}{\operatorname{F}}
\newcommand{\zd}{\operatorname{zd}}

\newcommand{\V}{\operatorname{V}}

\newcommand{\id}{\operatorname{id}}
\newcommand{\Ext}{\operatorname{Ext}}
\newcommand{\D}{\operatorname{D}}
\newcommand{\E}{\operatorname{E}}

\newcommand{\Supp}{\operatorname{Supp}}

\newcommand{\Hom}{\operatorname{Hom}}
\newcommand{\Soc}{\operatorname{Soc}}

\newcommand{\Ann}{\operatorname{Ann}}

\newcommand{\depth}{\operatorname{depth}}

\newcommand{\Char}{\operatorname{char}}
\newcommand{\coker}{\operatorname{coker}}
\newcommand{\HH}{\operatorname{H}}

\newcommand{\Se}{\operatorname{S}}
\newcommand{\q}{q}
\newcommand{\p}{p}
\newcommand{\f}{f}

\newcommand{\vpl}{\operatornamewithlimits{\varprojlim}}

\newcommand{\lo}{\longrightarrow}
\newcommand{\fm}{\frak{m}}
\newcommand{\fp}{\frak{p}}
\newcommand{\fq}{\frak{q}}
\newcommand{\fa}{\frak{a}}
\newcommand{\fb}{\frak{b}}

\newcommand{\fn}{\frak{n}}

\newcommand{\PP}{\mathbb{P}}

\begin{document}

\author[]{mohsen asgharzadeh}

\address{}
\email{mohsenasgharzadeh@gmail.com}

\title[ ]
{on the support of local and  formal cohomology}


\begin{abstract}
We compute support of formal cohomology modules in a serial of  non-trivial cases.
Applications are given.
For example, we compute injective dimension of certain local cohomology modules in terms  of dimension of their's support.
\end{abstract}

\maketitle

\section{introduction}

Throughout the introduction $(R, \fm,k)$ is a commutative noetherian regular local ring containing a field
 and  $I\lhd R$.
The notation $\D(-) := \Hom_R(-, \E_R(k))$ stands for the Matlis duality.
In view of \cite{li}
$\Supp(\D(\HH^i_I(R))) = \Spec(R)$ provided $\HH^i_I(R)
\neq0$ and $\Char R\neq 0$.

\begin{conjecture}\label{conj1}(Lyubeznik-Yildirim)
If $\HH^i_I(R)
\neq0$, then $\Supp(\D(\HH^i_I(R))) = \Spec(R)$.
\end{conjecture}

Let $R_n:=\mathbb{Q}[x_1,\ldots,x_n]_{(x_1,\ldots,x_n)}$ and  $\fp\in\Spec(R_n)$ be of height $h$.  In \S 3 we show $\Supp (\D(\HH^h_{\fp}(R_n))) = \Spec(R_n)$.
This yields
Conjecture \ref{conj1}  for some primes in $R_n$, see  Corollaries \ref{coprime}, \ref{nprime}  and \ref{6}. In Corollary \ref{cmo} we check Conjecture \ref{conj1} for any
 monomial ideal with respect to a regular sequence over Cohen-Macaulay rings.
In \S 4  we connect $\D(\HH^i_I(R))$  to topology of varieties. As an application, we  present Conjecture \ref{conj1} for locally complete-intersection (abb.  locally CI) prime ideals of $R_n$. Recall that $I$ is called locally CI if $I_{\fp}$ is CI for all $\fp\in \V(I)\setminus\{\fm\}$.

In the case of prime characteristic for the Cohen-Macaulay ideal $I\lhd R$,
Peskine and Szpiro proved  that $I$ is
 cohomologically CI in the sense that
$\HH^i_I(R)=0$ for all $i\neq\Ht(I)$. In particular, $\dim(\HH^i_ I(R)) = \id_R(\HH^i_ I(R)) $.

\begin{question}(Hellus, see \cite[Question 2.13]{Hel3})
Let  $I$ be Cohen-Macaulay. Is $\id_R(\HH^i_ I(R)) = \dim_R(\HH^i_ I(R))$ for any $i$?
\end{question}
We prove this in  dimension $5$. The same thing holds in dimension $6$ if the ring is essentially of finite type over $k$  (under the weaker generalized Cohen-Macaulay assumption), and in dimension $7$ if in addition $I$ is Gorenstein.

\begin{question}(Hellus, see \cite[Question 2.8]{Hel3})
 When does $\id(\HH^i_ I(R)) = \dim(\HH^i_ I(R))$ hold?
\end{question}
When $I$ is locally CI and equi-dimensional, we show $\id(\HH^i_ I(R)) = \dim(\HH^i_ I(R))$.  If  $\dim(R/I) = 2$ and $I$ is equi-dimensional  we show $\id(\HH^i_ I(R)) = \dim(\HH^i_ I(R))$. This drops  two technical assumptions of  \cite[Corollary 2.6(ii)]{Hel3}.

We present four
situations for which  $\f_I(R)=\q_I(R)=\Ht(I)$: i) $I$  is locally CI  and equi-dimensional, ii) $I$ is equi-dimensional and $\dim(R/I) = 2$, iii) $I$ is monomial (not necessarily squarefree) and $R/I$ is generalized Cohen-Macaulay, iv) $\dim R<7$ (or $\Char R>0$) and $R/I$ is generalized Cohen-Macaulay.
This is a local version of:
\begin{theorem}(Hartshorne, see \cite[Theorem III.5.2]{H4})
 Suppose $Y\subset \PP^n_k$ is nonsingular and of dimension $s$. Then $\p(U)=\q(U)=n-s-1$
 and in the prime characteristic case, it suffices to assume $Y$ is Cohen-Macaulay.
\end{theorem}
 We remark that the  prime characteristic (resp. square-free)  case  is due to  Peskine-Szpiro (resp. Richardson).

\section{Matlis dual of local cohomology}

The $i$-th local cohomology of $(-)$  with
respect to  $\fa\lhd R$ is
$\HH^{i}_{\fa}(-):={\varinjlim}_n\Ext^{i}_{R} (R/\fa^{n},
-)$. By $\Supp(M)$ we mean $\{\fp\in \Spec(R):M_{\fp}\neq 0\}$. If $M$
is finitely generated then  $\Supp(M)=\V(\Ann(M))$. The finitely generated assumption
 is really needed. However,   for any module $M$,  always  $\Supp(M) \subseteq \V(\Ann_R(M)) $  holds.
 By $\dim (-)$  (resp.  $\id(-)$) we mean  $\dim(\Supp(-))$ (resp. injective dimension).
We start by showing that the module version of Conjecture \ref{conj1} is false:

\begin{example}
There is a local regular ring $A$ and a finitely generated module $M$ such that $\HH^1_{\fm}(M)\simeq R/ \fm$, see e.g. \cite{li}.
It is enough to remark  that $\Supp(\D(\HH^1_\fm(M))=\Supp(\D(R/\fm))=\Supp(R/\fm)=\{\fm\}.$
\end{example}

\begin{observation}\label{nfg}
Let $(A,\fm,k)$ be a local ring, $M$ any module (not necessarily finitely generated)  and $\fa$ be an ideal. Suppose $\Supp(\D(\HH^i_{\fa}(M)))\supsetneqq \{\fm\}$.
 Then $\HH^i_{\fa}(M)$  is not finitely generated.
\end{observation}

\begin{proof}
Let $\fp\in\Supp(\D(\HH^i_{\fa}(M)))\setminus \{\fm\}$. Suppose on the contradiction that  $\HH^i_{\fa}(M)$  is  finitely generated. It turns out that  $\D( \HH^i_{\fa}(M))_{\fp}\simeq\Hom_{A_{\fp}}(\HH^i_{\fa}(M)_{\fp},0)=0$,
 a contradiction.
 \end{proof}

In the case $M$ is finitely generated the following is in \cite[Corollary 3.10]{Sch}.

\begin{observation}\label{sch}
	Let $A$ be a local ring, $M$ be such that
	$\D(\HH^i_{\fa}(M))=\fa\D(\HH^i_{\fa}(M))$. Then $\D(\HH^i_{\fa}(M))=0$.
\end{observation}

\begin{proof}
Let $\fb$ be such that $\fb\supseteq \fa$.
We look at
$\D(\HH^i_{\fa}(M))\supseteq\fb\D(\HH^i_{\fa}(M))\supseteq\fa\D(\HH^i_{\fa}(M))=\D(\HH^i_{\fa}(M))$
to see $\D(\HH^i_{\fa}(M))=\fb\D(\HH^i_{\fa}(M))\quad(+)$.
Suppose on the contradiction that $\D(\HH^i_{\fa}(M))\neq0$. Then $\HH^i_{\fa}(M)\neq0$.
Let $\fp\in\Ass(\HH^i_{\fa}(M))$. Consequently,  $\Hom_A(A/\fp,\HH^i_{\fa}(M))\neq0$.
Taking Matlis duality, $\D(\Hom_A(A/\fp,\HH^i_{\fa}(M)))\neq0$.
Now, we use the Hom-evaluation homomorphism
$$
\begin{array}{ll}
\frac{\D(\HH^i_{\fa}(M))}{\fp\D(\HH^i_{\fa}(M))}
&\simeq A/\fp\otimes_A \Hom_A(\HH^i_{\fa}(M),\E_A(k))\\
&\simeq\Hom_A(\Hom_A(A/\fp,\HH^i_{\fa}(M)),\E_A(k)),
\end{array}$$
	to see $\D(\HH^i_{\fa}(M))\neq\fp\D(\HH^i_{\fa}(M))$. Recall that $\fp\supset \fa$.
	By looking at $(+)$ we find a contradiction.\footnote{Second proof: Matlis proved in \cite[Corollary 4.10]{mat} that $\D(\HH^i_{\fa}(M))$ is complete with respect to $\fa$-adic topology. By induction,
$\D(\HH^i_{\fa}(M))=\fa^n\D(\HH^i_{\fa}(M))$. Then $\D(\HH^i_{\fa}(M))\simeq\D(\HH^i_{\fa}(M))^{\wedge_{\fa}}:=\vpl\frac{\D(\HH^i_{\fa}(M))}{\fa^n\D(\HH^i_{\fa}(M))}=0.$}
\end{proof}

The \textit{cohomological dimension} of $M$
with respect to $\fa$   is  the supremum of $i$'s such that $\HH^i_{\fa}(M)\neq 0$.  We denoted it by  $\cd(\fa, M)$, and we may use $\cd(\fa)$ instead of $\cd(\fa, R)$.
 It follows easily from \cite[1.2.1]{Hel2}  that
$\dim(\D(\HH^{\cd(\fa,M)}_{\fa}(M)))\geq \cd(\fa,M)$. This deduced from Grothendieck's vanishing theorem (GVT). Conversely,  it implies (GVT).
As an immediate application, we give  a new proof of \cite[Remark 2.5]{Hel1}:

\begin{corollary}\label{h}
Let $A$ be any noetherian ring and $M$ be any module (not necessarily finitely generated). Suppose $0<c:=\cd(\fa, M)$. Then $\HH^{c}_{\fa}(M)$  is not finitely generated.
\end{corollary}

\begin{proof}
We may assume $A$ is local. Recall that $\dim(\D(\HH^{c}_{\fa}(M)))\geq c>0$. By Observation \ref{nfg},
$\HH^{c}_{\fa}(M)$  is not finitely generated.
\end{proof}

\begin{remark}Example 2.1 is not in the spot of cohomological dimension.
Is $\Supp(\D(\HH^{\cd(\fa)}_{\fa}(R)))=\Spec(R)$?
Let $(R,\fm)$ be  complete and local with minimal prime ideals of different heights. Such a thing exists. Then $\Supp(\D(\HH^{\cd(\fm)}_{\fm}(R)))\neq\Spec(R)$.
Indeed, let $K:=\D(\HH^{\dim R}_{\fm}(R))$. It is shown by Grothendieck that $\Ass_R (K) =\{\fp\in \Spec (R): \dim R/\fp=\dim R\}$.
In particular, $\Ass_R (K)\subsetneqq \min(R)$, and so $\Supp(\D(\HH^{\cd(\fm)}_{\fm}(R)))\neq\Spec(R)$.
\end{remark}

Lyubeznik and Yildirim remarked that  the regular assumption in Conjecture \ref{conj1} is important.

\begin{proposition}\label{cm}
Let $(R,\fm)$ be an  analytically irreducible  local ring such that for any ideal $I$ one has $\Supp  (\D(\HH^i_I(R))) = \Spec(R)$ provided  $\HH^i_I(R)\neq 0$. Then $R$ is Cohen-Macaulay.
\end{proposition}

\begin{proof}
We may and do assume that $R$ is complete. Suppose on the contradiction that $R$ is not Cohen-Macaulay. There is an $i<\dim R$ such that
$\HH^i_{\fm}(R)\neq 0$.  This is proved  by Roberts that $\dim (R/\Ann\HH^i_{\fm}(R))\leq i$, see   \cite{rob}. It turns out that
$\Ann\HH^i_{\fm}(R)\neq 0$. We   deduce from this observation that $\Ann_R(\D(\HH^i_{\fm}(R)))\neq0$. Since $R$ is an integral domain,
$\V(\Ann_R(\D(\HH^i_{\fm}(R)))) \subsetneqq \Spec(R)$. In view of $\Supp(\D(\HH^i_{\fm}(R))) \subseteq \V(\Ann_R(\D(\HH^i_{\fm}(R)))) \subsetneqq \Spec(R)$ we find a contradiction that we search for it.
\end{proof}

\begin{fact} \label{top}(See \cite[Lemma 3.8]{Sch})
 Let $\{M_n\}$ be a decreasing family of
submodules of $M$. Suppose that the induced topology is equivalent with the $I$-adic topology on $M$. Then
${\vpl}_n\HH^{i}_{\fm}(\frac{M}{I^nM})={\vpl}_n\HH^{i}_{\fm}(\frac{M}{M_n})$.
\end{fact}

\begin{fact}(Formal duality, see e.g. \cite[Proposition 2.2.3]{O})\label{fd}
Let $A$ be a  Gorenstein local ring  of dimension $n$ and $p$ be the closed point. Let $X=\Spec(A)$ and $\mathfrak{X}$ denote the formal
completion of $X$ along  with $Y:=\V(I)$.  Then  $\HH^j_p(\mathfrak{X},\mathcal{O}_{\mathfrak{X}})\simeq\D(\HH^{n-j}_{I}(A))\simeq \vpl_{\ell}\HH^{ j}_{\fm}(A/I^{\ell})$ for any $j$.
\end{fact}
We close this section by presenting some motivational examples.  The following  extends \cite[Example 5.2]{Sch} and also a computation by Hellus.
 Proposition \ref{mo} (see below) drops the restriction on $|\min(I)|$.
\begin{example}
Let  $R:=k[[x_1,\ldots,x_d]]$ and  $I$ be a monomial ideal such that $\HH^{d-i}_I(R)\neq0$ and $|\min(I)|<3$. Then $\Supp(\D(\HH^{d-i}_I(R)))=\Spec(R)$.
\end{example}

\begin{proof}
Since $\D(\HH^{d-i}_{\sqrt{I}}(R))=\D(\HH^{d-i}_I(R))$ we may assume that $I$ is radical. First, assume that  $|\min(I)|=1$.
So that it is irreducible.
A monomial ideal is irreducible if and only if it is generated by pure powers of the variables. In particular it is complete-intersection.
In this case, the desired claim is due Hellus, see \cite[Lemma 2.1]{Hel2}.
The primary decomposition of $I$ is of the form
$I=(x_{i_1},\ldots,x_{i_t}) \cap(x_{j_1},\ldots,x_{i_s})=\fp\cap\fq$.  Set $\fp_n:=(x^n:x\in \fp)$.  By Mayer--Vietoris, $$\ldots\lo\HH^{i-1}_{\fm}(\frac{R}{\fp_n+\fq_n}) \lo \HH^{i}_{\fm}(\frac{R}{\fp_n\cap\fq_n})\lo\HH^{i}_{\fm}
(\frac{R}{\fp_n})\oplus\HH^{i}_{\fm}(\frac{R}{\fq_n})\lo\HH^{i}_{\fm}(\frac{R}{\fp_n+\fq_n})\lo\ldots$$
Note that $\frac{R}{\fp_n},\frac{R}{\fq_n}$, and $\frac{R}{\fp_n+\fq_n}$ are Cohen-Macaulay.
Also, $\dim\frac{R}{\fq_n}\neq\dim\frac{R}{\fp_n+\fq_n}\neq\dim\frac{R}{\fp_n} \quad(\dagger)$.
There are two possibilities: a) $\HH^{i}_{\fm}(\frac{R}{\fq_n})\neq0$ or $\HH^{i}_{\fm}(\frac{R}{\fq_n})\neq0$;
b) $\HH^{i}_{\fm}(\frac{R}{\fq_n})=\HH^{i}_{\fm}(\frac{R}{\fq_n})=0$.

a): The assumption  implies that $\dim(\frac{R}{\fq_n}) =i$ or $\dim(\frac{R}{\fq_n})=i$. Hence, $\dim(\frac{R}{\fp_n+\fq_n})\neq i$. Since   $\frac{R}{\fp_n+\fq_n}$ is Cohen-Macaulay, $\HH^{i}_{\fm}(\frac{R}{\fp_n+\fq_n})=0$.
Combine this along with the exact sequence to deduce that  the sequence $\HH^{i}_{\fm}(\frac{R}{\fp_n\cap\fq_n})\to\HH^{i}_{\fm}
(\frac{R}{\fp_n})\oplus\HH^{i}_{\fm}(\frac{R}{\fq_n})\to 0$ is exact. It yields the following exact  sequence $${\vpl}_n\HH^{i}_{\fm}(\frac{R}{(\fp\cap\fq)^n})\stackrel{\ref{top}}\simeq{\vpl}_n\HH^{i}_{\fm}(\frac{R}{\fp_n\cap\fq_n})\lo{\vpl}_n\HH^{i}_{\fm}
(\frac{R}{\fp_n})\oplus{\vpl}_n\HH^{i}_{\fm}(\frac{R}{\fq_n})\lo 0.$$ By \cite[Lemma 2.1]{Hel2},
 either $\Supp(\D(\HH^{i}_{\fm}(\frac{R}{\fq_n})))=\Spec(R)$ or $\Supp(\D(\HH^{i}_{\fm}(\frac{R}{\fp_n})))=\Spec(R)$.
So, $$\Supp({\vpl}_n\HH^{i}_{\fm}(\frac{R}{(\fp\cap\fq)^n})) \supset\Supp({\vpl}_n\HH^{i}_{\fm}
(\frac{R}{\fp_n}))\cup\Supp({\vpl}_n\HH^{i}_{\fm}(\frac{R}{\fq_n}))=\Spec(R).$$From this we get that $\Supp({\vpl}_n\HH^{i}_{\fm}(\frac{R}{(\fp\cap\fq)^n}))=\Spec(R)$.
It remains to recall   that $ {\vpl}_n\HH^{i}_{\fm}(\frac{R}{(\fp\cap\fq)^n}))=\D(\HH^{d-i}_I(R))$.

b): Since $\HH^{i}_{\fm}(\frac{R}{\fq_n})=\HH^{i}_{\fm}(\frac{R}{\fq_n})=0$, we get that$${\vpl}_n\HH^{i-1}_{\fm}
(\frac{R}{\fp_n})\oplus{\vpl}_n\HH^{i-1}_{\fm}(\frac{R}{\fq_n})\lo{\vpl}_n\HH^{i-1}_{\fm}(\frac{R}{\fp_n+\fq_n}) \lo {\vpl}_n\HH^{i}_{\fm}(\frac{R}{\fp_n\cap\fq_n})\lo 0.$$
Since
$\HH^{d-i}_I(R)\neq0$, we have $0\neq\D(\HH^{d-i}_I(R))\simeq\D(\HH^{d-i}_{I}(\frac{R}{\fp_n\cap\fq_n}))$, because $\rad (\fp_n\cap\fq_n)=\rad (\fp\cap \fq)=I$. This implies that  ${\vpl}_n\HH^{i}_{\fm}(\frac{R}{\fp_n\cap\fq_n})\neq0$. From the displayed exact sequence,  ${\vpl}_n\HH^{i-1}_{\fm}(\frac{R}{\fp_n+\fq_n})\neq 0$. We call this property by $(\ast)$. Keep in mind that $\fp_n+\fq_n$  is generated by powers of variables. Suppose first that ${\vpl}_n\HH^{i-1}_{\fm}
(\frac{R}{\fp_n})={\vpl}_n\HH^{i-1}_{\fm}
(\frac{R}{\fp_n})= 0$. It follows that ${\vpl}_n\HH^{i-1}_{\fm}(\frac{R}{\fp_n+\fq_n})\simeq {\vpl}_n\HH^{i}_{\fm}(\frac{R}{\fp_n\cap\fq_n})$. From this, we get the claim, because $\Supp({\vpl}_n\HH^{i-1}_{\fm}(\frac{R}{\fp_n+\fq_n}))=\Spec(R)$. Second, we may assume that  ${\vpl}_n\HH^{i-1}_{\fm}
(\frac{R}{\fp_n})\oplus{\vpl}_n\HH^{i-1}_{\fm}(\frac{R}{\fq_n})\neq 0$. Without loss of the generality we assume that
$\HH^{i-1}_{\fm}
(\frac{R}{\fp_n})\neq 0$. Recall that $\frac{R}{\fp_n}$ is Cohen-Macaulay. This implies that $\dim(\frac{R}{\fp_n})=i-1$. We get from $(\ast)$ that $\dim (\frac{R}{\fp_n+\fq_n})=i-1$, because $\frac{R}{\fp_n+\fq_n}$ is Cohen-Macaulay.  To get  a contradiction
it remains to apply $(\dagger)$.
\end{proof}

 \begin{example}
Let  $R:=k[[x_1,\ldots,x_6]]$  and $I := (x_1, x_2)\cap (x_3, x_4)\cap(x_5, x_6)$. Then $\Supp(\D (\HH^{3}_I(R)))=\Spec(R)$.
\end{example}

\begin{proof}
It turns out from Mayer--Vietoris that
$\HH^{3}_I(R)\twoheadrightarrow\HH^4_{(x_1,x_2,x_5,x_6)}(R)\oplus\HH^4_{(x_3,x_4,x_5,x_6)}(R)\to 0$   is exact (see the third item in \cite[Example 2.11]{Hel1}).
Thus, $\D (\HH^4_{(x_1,x_2,x_5,x_6)}(R))\oplus\D (\HH^4_{(x_3,x_4,x_5,x_6)}(R))\hookrightarrow\D (\HH^{3}_I(R)).$
 By the above example, $\Supp(\D (\HH^{4}_{(x_1,x_2,x_5,x_6)}(R)))=\Spec(R).$ In view of   $\Spec(R)\subseteq\Supp(\D (\HH^4_{(x_1,x_2,x_5,x_6)}(R)))\subseteq\Supp(\D (\HH^{3}_I(R)))\subseteq\Spec(R),$ we see $\Supp(\D (\HH^{3}_I(R)))=\Spec(R)$.
\end{proof}

\section{computing support of  formal cohomology}

\begin{observation}\label{cdmodule}We observed that the module version of Conjecture 1.1 is not true.
Let $R$ be an integral domain such that   $\Supp_R  (\D(\HH^{\cd(I)}_I(R))) = \Spec(R)$. If $M$ is  finitely generated and torsion-free,
 then  $\Supp (\D(\HH^{\cd(I,M)}_I(M)))=\Spec(R).$ Indeed, one has $\Supp_R  (M) = \Spec(R)$. This implies that  $\cd(I,R)=\cd(I,M)$, see e.g. \cite[Lemma 2.1]{Sch}. Since $\HH^{\cd(I,R)}_I(-)$ is right exact,  $
\D(\HH^{\cd(I)}_I(M))\simeq\D(\HH^{\cd(I)}_I(R)\otimes_RM)\simeq \Hom_R(M,\Hom_R(\HH^{\cd(I)}_I(R),\E_R(k)))$.
Also, $
\Ass(\Hom(M,\D(\HH^{\cd(I)}_I(R))))
=\Supp(M)\cap \Ass(\D(\HH^{\cd(I)}_I(R)))=\Ass(\D(\HH^{\cd(I)}_I(R)))
$. This yields
$0\in\Ass \D(\HH^{\cd(I)}_I(M))$. So, $\Supp(\D(\HH^{\cd(I)}_I(M))) = \Spec(R).$
\end{observation}

By $(HLVT)$ we mean Hartshorne-Lichtenbaum Vanishing Theorem.

\begin{proposition}\label{3}
 Conjecture \ref{conj1} is true if $\dim R<4$.
 \end{proposition}

\begin{proof}
We deal with the case $\dim R=3$.
We may assume $I$ is radical.
First we deal with the case  $\Ht(I)=1$ and that $I$ is height unmixed. It turns out that $I=(x)$ for some $x$. It is proved in \cite[Thorem 1.3]{Hel3}  that $\Ass(\D(\HH^1_{(x)}(R)))=\Spec(R)\setminus\V(x)$, so that $0\in \Ass(\D(\HH^1_I(R)))$. Consequently,
$\Supp(\D(\HH^1_I(R)))=\Spec(R)$. Now, we assume $I$ is any height one ideal.
We are going to reduced to the unmixed case. Write $I=J\cap K$  where $J$  is
 unmixed, $\Ht (J) = 1$, $\Ht (K) \geq 2$, and $\Ht(J + K) \geq 3$. In the light of  Mayer-Vietoris
there is an exact sequence $\HH^i_{J + K}(R)\to\HH^i_{J}(R) \oplus\HH^i_{K}(R)\to \HH^i_{J \cap K}(R)\to\HH^{i+1}_{J + K}(R).$
Note that $\HH^1_{J + K}(R)=\HH^1_{K}(R)=\HH^{2}_{J + K}(R)=0$. From this,
$\D(\HH^1_{I}(R))\simeq\D(\HH^1_{J}(R))$ and by the height unmixed case $\Supp(\D(\HH^1_I(R)))=\Spec(R)$.
After taking Matlis duality from
$ \HH^2_{I}(R)\to\HH^{3}_{\fm}(R)\to 0$ we have $\widehat{R}\simeq\D(\HH^{3}_{\fm}(R))\subset\D(\HH^2_{I}(R))$. This implies that
$\Supp(\D(\HH^2_I(R)))=\Spec(R)$.
Now, we deal with the case  $\Ht(I)=2$. By $(HLVT)$, $\cd(I)<3$.  Recall that $2=\Ht(I)\leq\cd(I)<3$. From this we deduce that
$\HH^1_I(R)=\HH^3_I(R)=0$. \begin{enumerate}
\item[Fact] A): Let $A$ be an analytically irreducible local ring of dimension $d$ and $J$ be of dimension one.
 Then $\Supp_A (\D(\HH^{d-1}_J(A))) = \Spec(R)$. Indeed, since $\widehat{A}$  is domain and  due to $(HLVT)$ we have $\HH^d_J(A)=0$.
It is proved in \cite[Thorem 2.4]{Hel3}   that if $J$ is $1$-dimensional and that $\HH^d_J(A)=0$, then
$\Assh(\D(\HH^{d-1}_J(A)))=\Assh(A)$. Apply this, we get that
$\Supp(\D(\HH^{d-1}_J(A)))=\Spec(A)$.
\end{enumerate}
In the light of this fact we see $\Supp(\D(\HH^{2}_I(R)))=\Spec(R)$.
Finally,  we assume $\Ht(I)=3$. We have  $\HH^i_{I}(R)=0$ for all $i<\depth (R)=3$ and  $\Supp(\D(\HH^3_{I}(R)))=\Spec(R)$ because $\HH^3_{I}(R)=\D(\E_R(R/\fm))=\widehat{R}$.
\end{proof}

\begin{lemma}\label{CD1}
Let $R$ be a regular local ring containing a field  and $I$ be an unmixed ideal such that $\cd(I)=\dim R-1$.   Then $\Supp  (\D(\HH^{\cd(I) }_I(R))) = \Spec(R)$.
\end{lemma}

\begin{proof} Set $d:=\dim R$.
In view of   Fact A) in Proposition \ref{3} we can assume that $\Ht(I)<d-1$.
 Let $\fp\in \V(I)$ be a prime ideal of height $d-1$. Since $I$ is unmixed, $I_{\fp}$ is not primary to $\fp R_{\fp}$. Due to $(HLVT)$,  $\HH^{d-1}_I(R)_{\fp}\simeq\HH^{d-1}_{I_{\fp}}(R_{\fp})=0$.
This means that  $\dim(\HH^{d-1}_I(R))=0$.  We may assume that $\Char R=0$. By a result of Lyubeznik $0\leq\id(\HH^{d-1}_I(R))\leq\dim(\HH^{d-1}_I(R))=0$, see \cite[Corollary 3.6]{lin}.
This implies that $\HH^{d-1}_I(R)\simeq\E_R(R/\fm)^t$. By  Matlis theory, we have $\widehat{R}^t\simeq\D(\HH^{d-1}_{I}(R))$, and so
$\Supp(\D(\HH^{d-1}_I(R)))=\Spec(R)$. \end{proof}

\begin{corollary}\label{4one}
Let $R$ be a regular local ring of dimension $4$ containing a field   and $I$ be an  ideal of height $2$ and $\cd(I)=3$.    Then $\Supp (\D(\HH^{3}_I(R))) = \Spec(R)$.
\end{corollary}

\begin{proof}Without loss of the generality, $I=\rad(I)$.
We may assume that  $I$  is not unmixed, see Lemma \ref{CD1}. Write $I=J\cap K$  where $J$  is radical and
 unmixed, $\Ht (J) = \Ht(I)$. Also, $K$ is radical, $\Ht (K) > \Ht(I)=2$ and $\Ht(J + K) = 4$. We can assume $K\neq \fm$, i.e, $\dim R/K>0$.  In the light of  Mayer-Vietoris
there is an exact sequence  $\HH^{3}_{J \cap K}(R)\to\HH^{4}_{J + K}(R)\to\HH^{4}_{J}(R) \oplus\HH^{4}_{K}(R).$ Due to $(HLVT)$,  $\HH^{4}_{J}(R)=\HH^{4}_{K}(R)=0$.
We apply Matlis functor  to see $\widehat{R}\simeq\D(\HH^{4}_{\fm}(R))\subset\D(\HH^3_{I}(R))$. This implies that
$\Supp(\D(\HH^3_I(R)))=\Spec(R)$.
\end{proof}

For any module $M$ we have $M\hookrightarrow\D(\D(M))$. From this, $\Ann(\D^2(M))\subset\Ann(M)$. Since $\Ann(\D(M))\subset\Ann(\D^2(M))\subset\Ann(M)\subset\Ann(\D(M))$, one has $\Ann(M)=\Ann(\D(M))$.

\begin{proposition}\label{simon}
Let $R:=\mathbb{Q}[x_1,\ldots,x_n]_{(x_1,\ldots,x_n)}$ and  $\fp\in\Spec(R)$ be of height $h$.  Then $\Supp (\D(\HH^h_{\fp}(R))) = \Spec(R)$.
\end{proposition}

\begin{proof} Since the proposition is trivial for the maximal ideal,
we may assume that $\dim R/ \fp>0$.
The proof divided into some steps.

 \begin{enumerate}
\item[Step ]1) Let $A$ be a countable  local ring, $J\vartriangleleft A$ and $M$  a
complete $A$-module with respect to $J$-adic topology. Suppose $\V(J)\subset\Supp(M).$ Then $\Supp(M) = \V(\Ann_AM)$.
\end{enumerate}

Here, we use a trick taken from \cite[9.3]{Si}. Recall that  $\Supp(M) \subseteq \V(\Ann_AM) $. Suppose on the contradiction that there is $\fp\in \V(\Ann_AM)\setminus\Supp(M)$.
For each $y\in \fm\setminus\fp$,  set $ K_y := \ker (M\stackrel{y}\lo M)$. Since $M_{\fp}=0$, each   element of $M$ is annihilated by some $x\in A\setminus\fp$. In conclusion, 
$M= \bigcup_{y\in \fm\setminus\fp}K_y$. Any $A$-module homomorphism is continuous. Apply this for the
multiplication map, we see $K_y$ is closed in $M$. We recall Baire's category theorem: any complete metric space
is not the union of a countable family of nowhere dense subsets. By definition, a topological set is called  nowhere dense  if its closure has no interior points. Therefore, one of $\{K_y\}_{y\in \fm\setminus\fp}$, say $K_x$, has an interior point, because $K_x$ is closed.
Since the
translations $M\to M$ are bicontinuous, we may and do assume that the origin $0$ is an interior point of $K_x$.
That is
$K_x \supset J^n M$ for some $n$. This in turns  equivalent with $J^nx\subset\Ann_AM\subseteq \fp.$ We know from $x\notin\fp$ that
$J\subset \fp$. This is a contradiction, because $\fp\in\V(J)\subset\Supp(M)$.

 \begin{enumerate}
\item[Step]2)  Let $I\vartriangleleft R$ be  such that  $\V(I)\subset\Supp_R  (\D(\HH^i_I(R)))$.
Then $\Supp  (\D(\HH^i_I(R))) = \Spec(R)$.
\end{enumerate}

Indeed, let $\fp$ be a minimal element in support of $\HH^i_I(R)$.   By \cite[Corollary 3.6]{lin}
$\id(\HH^i_{I_{\fp}}(R_{\fp}))\leq\dim(\HH^i_{I_{\fp}}(R_{\fp}))=\dim(\HH^i_I(R)_{\fp})=0.$ This means that $ \HH^i_{I_{\fp}}(R_{\fp})=\E_{R_{\fp}}(R_{\fp}/\fp R_{\fp} )^t$. In particular, $ \HH^i_{I_{\fp}}(R_{\fp})$ is faithful as an $R_{\fp}$-module. Since $R\subset R_{\fp}$, it makes $ \HH^i_{I_{\fp}}(R_{\fp})$ faithful  over $R$. Let
 $r$ be such that $r\HH^i_I(R)=0$. Clearly, $r\HH^i_{I_{\fp}}(R_{\fp})=0$. From this, $r=0$.
 We   deduce that $\Ann_R(\D(\HH^i_I(R)))=\Ann_R(\HH^i_I(R))=0$.   Matlis proved in \cite[Corollary 4.10]{mat} that
$\D(\HH^i_I(R))$ is complete with respect to $I$-adic topology.  Eventually, Step 1) implies that  $\Supp  (\D(\HH^i_I(R))) = \Spec(R)$, as claimed.

 \begin{enumerate}
\item[Step] 3) Let $A$  be any Gorenstein local ring and  $\fq$ a prime ideal of grade $g$. Then  $\Ass(\HH^g_{\fq}(A))=\{\fq\}$.
\end{enumerate}
 This follows by looking at the explicit injective resolution of $A$: Let $\triangle_i$ be the prime ideals of height $i$. The $h$-th spot of the resolution is
$\bigoplus_{\fp\in\triangle_h}\E_A(A/ \fp)$.  If we apply $\Gamma_{\fq}$ to it we see that $\HH^g_{\fq}(A)\hookrightarrow\E_A(A/ \fq)$. Since
$\Ass(\E_A(A/ \fq))=\{\fq\}$, we have $\Ass(\HH^g_{\fq}(A))=\{\fq\}$.
 \begin{enumerate}
\item[Step]4) Since $\zd(\HH^h_{\fp}(R))=\cup_{\fq\in\Ass(\HH^h_{\fp}(R))}\fq=\fp$ we see $\HH^h_{\fp}(R)\hookrightarrow\HH^h_{\fp}(R)_{\fp}$.
Let $C:=\frac{\HH^h_{\fp}(R)_{\fp}}{\HH^h_{\fp}(R)}$. Then $\fp C\neq0$.
\end{enumerate}
Indeed,
suppose on the contradiction that $\fp C=0$. This in turns equivalent with $\fp\HH^h_{\fp}(R)_{\fp}\subset\HH^h_{\fp}(R)$.
Since $\HH^h_{\fp_{\fp}}(R_{\fp})\simeq\E_{R}\left(R/ \fp \right)_{\fp}\simeq\E_{R}\left(R/ \fp \right)$,  we  arise to the  injectivity of $\HH^h_{\fp_{\fp}}(R_{\fp})$. Hence
it is divisible over the integral domain. This implies that $\HH^h_{\fp}(R)_{\fp}=\fp\HH^h_{\fp}(R)_{\fp}\subset\HH^h_{\fp}(R)\subset\HH^h_{\fp}(R)_{\fp}.$
We deduce from this that $\HH^h_{\fp}(R)\simeq\HH^h_{\fp}(R)_{\fp}$. Also, we know from  Fact A)  in Corollary \ref{lci} (see below) that   $\dim_R (\HH^h_{\fp}(R))=\id_R (\HH^h_{\fp}(R))=\dim R/ \fp$. Recall that $\dim R/ \fp>0$. To get a contradiction it remains to note that $\id_R(\HH^h_{\fp}(R)_{\fp})=0$. In sum,  $\fp C\neq0$.

Step 5)
 We look at the exact sequence
$0\to \HH^h_{\fp}(R)\stackrel{\iota}\lo \HH^h_{\fp}(R)_{\fp}\to C:=\coker(\iota)\to 0$
and take Matlis duality to see
$0\to \D(C)\to \D(\HH^h_{\fp}(R)_{\fp})\to \D(\HH^h_{\fp}(R)) \to 0.$
This shows that $$\Supp(\D(\HH^h_{\fp}(R)_{\fp}))\subset\Supp(\D(\HH^h_{\fp}(R)))\cup\Supp(\D(C))\quad(\ast)$$
\begin{enumerate}
\item[Claim] A): $\fp\notin\Supp(\D(C))$. Indeed, since $\Supp(-)\subset\V(\Ann(-))$ we need to
show  $\fp\notin\V(\Ann(\D(C)))$. Since
 $\V(\Ann(\D(C)))=\V(\Ann(C))$, it remains to apply Step 4).
\end{enumerate}

Since  $\HH^h_{\fp}(R)_{\fp} $ is injective, $\D(\HH^h_{\fp}(R)_{\fp})$ is flat. Support of any flat
module is the full prime spectrum. In particular, $\fp\in \Supp(\D(\HH^h_{\fp}(R)_{\fp}))$. We apply Claim A) along with $(\ast)$
to see that $\fp\in\Supp(\D(\HH^h_{\fp}(R)))$. We deduce that  $\V(\fp)\subset\Supp (\D(\HH^h_{\fp}(R)))$. Recall that $\D(\HH^h_{\fp}(R))$
is complete with respect to $\fp$-adic topology.
Now,  Step 2) establishes what we wanted.
\end{proof}

\begin{corollary}\label{coprime}
Let $R_n:=\mathbb{Q}[x_1,\ldots,x_n]_{(x_1,\ldots,x_n)}$    and $\fp\in \Spec(R_n)$ be cohomologically CI.    Then $\Supp (\D(\HH^{i }_{\fp}(R_n))) = \Spec(R_n)$ provided
$\HH^{i }_{\fp}(R_n)\neq 0$.
\end{corollary}

\begin{corollary}\label{nprime}
Let $R:=R_n$    and $\fp\in \Spec(R)$ be of height at least $n-2$.    Then $\Supp (\D(\HH^{i }_{\fp}(R))) = \Spec(R)$ provided
$\HH^{i }_{\fp}(R)\neq 0$. In particular, for any  $\fp\in \Spec(R_4)$ we have $\Supp (\D(\HH^{i }_{\fp}(R_4))) = \Spec(R_4)$ provided
$\HH^{i }_{\fp}(R_4)\neq 0$.
\end{corollary}

\begin{proof}We may assume that $n>3$ (see Proposition \ref{3}).  In view of Fact \ref{3}.A) we may and do assume that $\Ht(\fp)=n-2$. If $\cd(\fp)=n-2$, then  $\fp$ is cohomologically CI. The desired claim is Corollary \ref{coprime}.  By
$(HLVT)$, $\cd(\fp)\neq n$. Without loss of the generality, we may assume that $\cd(\fp)=n-1$ (this may rarely happen, see e.g. the second vanishing theorem \cite[Theorem 7.5]{Hcd}).  Since $\HH^{i }_{\fp}(R)\neq 0$ we deduce that either $i=n-2$ or $i=n-1$.
Due to Proposition \ref{simon} $\Supp (\D(\HH^{n-2}_{\fp}(R))) = \Spec(R)$.
In view of  Lemma \ref{CD1}  $\Supp (\D(\HH^{n-1}_{\fp}(R))) = \Spec(R)$. This completes the proof of first claim.
Here, we deal with the particular case: The cases $\Ht(\fp)\in\{1,3,4\}$ are trivial, see Proposition \ref{3}.  We may assume that  $\Ht(\fp)=2$.
We proved this in the first part.
\end{proof}

\begin{example}
 Let $R:=R_n$    and $\fp\in \Spec(R)$ be of height $n-2$.
Suppose in addition that there is a homogeneous $\fq\lhd_{h} A$ such that $\fp=\fq R$, where $A:=\mathbb{Q}[x_1,\ldots,x_n]$.
It follows from primeness that $X:=\Proj(R/\fq)$ is connected and that $\dim X=1$. By second vanishing theorem, $\HH^{n-2}(\PP^{n-1}_{\mathbb{Q}}\setminus X,\mathcal{O})=0$.
We use this to see $\HH^{n-1}_{\fq}(A)=0$. Since localization commutes with local cohomology, we deduce that
$\HH^{n-1}_{\fp}(R)=0$.
\end{example}

Let  $R:=k[x_1,\ldots,x_d]$ be the polynomial ring. Let $q>1$ be any integer. The assignment $X_i\mapsto X_i^q$ induces a ring homomorphism
$\F_q:R\to R$. By $\F_q(R)$, we mean $R$ as a group equipped with  left and right scalar multiplication from $R$
given by
$a.r\star b = a\F_q(b)r,$ where $a,b\in R$ and $r\in \F_q(R).$  For an $R$-module $M$, set $\F_q(M):= \F_q(R)\otimes M.$
The left $R$-module structure of $\F_q(R)$ endows $\F_q(M)$ with a left $R$-module structure
such that $c(a\otimes x)=  (ca)\otimes x$. For an $R$-linear map  $\varphi: M\to N$, we set
$\F_q(\varphi):=1_{\F_q(R)}\otimes\varphi$. For any monomial ideal $I$ we have $\F_q( R/I)\simeq R/I^{[q]}$.
Also,  $\F_q(R)$ is flat.

\begin{proposition}\label{mo}
Let  $R$   and $I$ be as above. If $\HH^{i}_I(R)\neq0$, then $\Supp(\D(\HH^{i}_I(R)))=\Spec(R)$.
\end{proposition}

\begin{proof} Here, $\D(-) := \Hom_R(-, \E_R(\frac{R}{(x_1,\ldots,x_d)}))$. We use a trick taken from \cite{li}. We may replace $I$ by its radical.  A monomial ideal   is   radical  if and only if it  is   squarefree monomial ideal.
We assume that  $I$ is squarefree.

Step 1) One has $(1.a):\Ext^j_R(R/I,R)\subset \Ext^j_R(R/I^{[q]},R)$. This is in \cite[Theorem 1.(i)]{lim}.
  Let $\mathcal{F}$ be a finite free resolution of  $R/I$. It turns out  $\F_q(\mathcal{F})$ is a free resolution of $R/I^{[q]}$.
We have, (1.b): $$
\F_q(\Ext^j_R(\frac{R}{I},R))\simeq \F_q(H^j(\Hom_R(\mathcal{F},R)) \simeq H^j(\F_q(\Hom_R(\mathcal{F},R)))\simeq H^j(\Hom_R(\F_q(\mathcal{F}),R))\simeq \Ext^j_R(\frac{R}{I^{[q]}},R)
.$$
Let $J$ be a monomial ideal. In the same vein, we have $
\F_q(\Ext^j_R(R/J,R/I))\simeq \Ext^j_R(R/ J^{[q]},R/I^{[q]})
$. Since $\F_q$ computes with taking direct limit and by the previous isomorphism, we have
$$\F_q({\varinjlim}_n\Ext^{j}_{R}(\frac{R}{\fm^n},\frac{R}{I}))\simeq{\varinjlim}_n\F_q(\Ext^{j}_{R}(\frac{R}{\fm^n},\frac{R}{I}))\simeq {\varinjlim}_n\Ext^{j}_{R}((\frac{R}{\fm^n})^{[q]},\frac{R}{I^{[q]}})\simeq\HH^{j}_{\fm}(\frac{R}{I^{[q]}})\quad(1.c)
$$

Step 2) There is an artinian module $N$ such that$$
\begin{CD}
\D(\HH^{i}_I(R)) \simeq\vpl (N  @<\alpha<< \F_q(N) @< \F_q(\alpha)<<F_{q^2}(N)  @<F_{q^2}(\alpha)<< \cdots),
\end{CD}
$$with surjective morphisms. Indeed, first we denote the local duality by $(+)$ and remark that:
$$\F_q(\D(\Ext^j_R(\frac{R}{I},R)))
 \stackrel{(+)}\simeq \F_q(\HH^{d-j}_{\fm}(\frac{R}{I}))
 \stackrel{(1.c)}\simeq \HH^{d-j}_{\fm}(\frac{R}{I^{[q]}})
 \stackrel{(+)}\simeq \D(\Ext^j_R(\frac{R}{I^{[q]}},R))
 \stackrel{(1.b)}\simeq\D(\F_q(\Ext^j_R(\frac{R}{I},R)))\quad(\ast)
$$
We look at $
\D(\HH^{i}_I(R))
\stackrel{\ref{fd}}\simeq\vpl_{\ell}\HH^{d-i}_{\fm}(\frac{R}{I^{\ell}})\stackrel{\ref{top}}
\simeq\vpl_{\ell}\HH^{d-i}_{\fm}(\frac{R}{I^{[q^{\ell}]}})\stackrel{(+)}\simeq\vpl_{\ell}\D(\Ext^{i}(\frac{R}{I^{[q^{\ell}]}},R))
\stackrel{(\ast)}\simeq\vpl_{\ell}F_{q^{\ell}}(\D(\Ext^{i}(\frac{R}{I},R)).
$  In view of Step (1.a) the maps $\Ext^j_R(\frac{R}{I^{[q^{\ell}]}},R)\to \Ext^j_R(\frac{R}{I^{[q^{\ell+1}]}},R)$
are injective. So, their Matlis dual are surjective. It remains to  set $N:=\D(\Ext^{i}(R/I,R))$.

Step 3) Let $J:=\Ann_R(N)$. We are going to show that $J$ is monomial. Recall that
 $\Ann_R(N)=\Ann_R(\D(N))$.  Also,  $\D(N)=\Ext^{i}_R(R/I,R)(g)$ where $g\in \mathbb{Z}^d$ is a graded shifting.  A minimal free-resolution $ \mathcal{F}$ of $R/I$ is a complex consisting of $\mathbb{Z}^d$-graded modules equipped with
$\mathbb{Z}^d$-graded differentials.  From this,
$\HH^{\ast}(\Hom(\mathcal{F},R))$ is  $\mathbb{Z}^d$-graded. Annihilator of any $\mathbb{Z}^d$-graded module is an $\mathbb{Z}^d$-graded  ideal.
So, $J$ is monomial.

Step 4) There exists an element $\textbf{n}:= (n_0,n_1, \ldots)\in \D(\HH^{i}_I(R))$
such that $n_k \in F_{q^k}(N)$ and $F_{q^{k-1}}(n_k) = n_{k-1}$
and with the property that
$\Ann(n_k)\subset \fm^k$
for all $k \geq 4$. This step is in the proof of \cite[Theorem 1.1]{li}. For the simplicity of the reader,
we repeat the main idea:  Since $J$ is monomial, $\Ann(\F_q(N)) = J^{[q]}$. In particular, $\ker( \F_q(N)\stackrel{\alpha}\lo N) \neq 0.$ Let $b_1 \in\Soc(\ker(\alpha))$,  and define $b_k$  as the image of $b_{k-1}$
under $\F_q$. Let $n_0\in N$. By the subjectivity, there are $n_i$ such that
$n_{i}=F_{q^i}(\alpha)(n_{i+1})$. For each $k>3$, they showed that
 either $\Ann(n_k)\subset \fm^k$
or $\Ann(n_k + b_k)\subset \fm^k$.
In the second possibility we replace $n_i$ with $n_k + b_k$ and denote it again by $n_i$. This is now clear that $\Ann(n_k)\subset \fm^k$ for all $k \geq 4$.

Step 5) Here, we present the proof: Let $r\in\Ann(\textbf{n})$. Then $r\in\bigcap\Ann(n_k)\subseteq\bigcap \fm^k=0$.
From this, we get that $0=(0:\textbf{n})\in\Ass(\D(\HH^{i}_I(R)))$. Consequently,  $\Supp(\D(\HH^{i}_I(R)))=\Spec(R)$.
\end{proof}

Hellus proved  that $\Supp(\D(\HH^{\cd(I)}_I(R)))=\Spec(R)$ over any  equi-characteristic analytically irreducible local ring provided $I$ generated by a regular sequence.
We extend this  in the following sense:

\begin{corollary}\label{cmo}
Let  $(R,\fm,k)$ be a Cohen-Macaulay equi-characteristic analytically irreducible local ring, $I$ a monomial ideal with respect to a regular sequence such that $\HH^{i}_I(R)\neq0$. Then $\Supp(\D(\HH^{i}_I(R)))=\Spec(R)$.
\end{corollary}

\begin{proof}We may assume that $R$ is complete.
Let $\underline{x}$ be a regular sequence such that $I$ is a monomial ideal with respect to $\underline{x}$. We extend $\underline{x}$ to a full
parameter sequence $x_1,\ldots,x_d$ of $R$.
 By Cohen's structure theorem, $R$ is module-finite over $R_0:=k[[x_1,\ldots,x_d]]$. Since $R$ is Cohen-Macaulay and by applying Auslander-Buchsbaum formula, the extension
$R_0\subset R$ is free. Note that $I=JR$ for some monomial $J\lhd R_0$. Recall that $\HH^{i}_I(R)\simeq\HH^{i}_J(R_0)\otimes_{R_0}R$, because $R$ is flat over $R_0$. Now, we use
 \cite[Page 24]{Hel3} to see $\Hom_R(\HH^{i}_I(R),\E_{R}(k))
\simeq\Hom_{R_0}(R,\Hom_{R_0}(\HH^{i}_J(R_0),\E_{R_0}(k))).$
By the proof of Proposition \ref{mo}, $R_0\hookrightarrow\D_{R_0}(\HH^{i}_J(R_0))$. Apply $\Hom_{R_0}(R,-)$ to $R_0\hookrightarrow\D_{R_0}(\HH^{i}_J(R_0))$ we have$$\Hom_{R_0}(R, R_0)\hookrightarrow\Hom_{R_0}(R,\D_{R_0}(\HH^{i}_J(R_0)))\simeq\Hom_R(\HH^{i}_I(R),\E_{R}(k)).$$

 One has $0\in\Ass_R(\Hom_{R_0}(R, R_0))$.
Indeed,  recall that $\Ass_{R_0}(\Hom_{R_0}(R, R_0))=\Supp_{R_0}(R)\cap\Ass_{R_0}(R_0)=\{0\}$.
Let $r\in R$. Multiplication by $r$ defines  a map $\mu(r)\in\Hom_{R_0}(R,R)$. Let $\rho:R\to R_0$ be the splitting map.  We define $\varphi:R\to \Hom_{R_0}(R, R_0)$ by
$\varphi(r):=\rho(\mu(r))$. This is an $R$-homomorphism. Let $\fa:=\ker(\varphi)\lhd R$. Then $\frac{R_0}{\fa\cap R_0}\hookrightarrow\frac{R}{\fa}\hookrightarrow\Hom_{R_0}(R, R_0)$.
Hence $\Ass_{R_0}(\frac{R_0}{\fa\cap R_0})\subset\Ass_{R_0}(\Hom_{R_0}(R, R_0))=\{0\}.$ This means that $\fa\cap R_0=0$. Since  the extensions are integral
we see $\dim(R)=\dim(R_0)=\dim(\frac{R_0}{\fa\cap R_0})=\dim(\frac{R}{\fa}).$ Consequently, $\fa=0$.
In view of $R\hookrightarrow\Hom_{R_0}(R, R_0)\hookrightarrow\Hom_R(\HH^{i}_I(R),\E_{R}(k))$, we see
 $\Supp(\D(\HH^{i}_I(R)))=\Spec(R)$.
\end{proof}

\section{A connection  to topology of varieties}

\begin{discussion}\label{1}     Let $k$ be a field and $ A := k[[t_1, \ldots,t_n]]$.
Let  $Y:=\V(I)$ be a closed subset of $X := \Spec (A)$, defined by an ideal $I$. Let $\mathfrak{X}$ be the formal completion of $X$ along $Y$ and let $p$ be the closed point. Here $\mathcal{H}^i_p(Y)$ denotes the local \textit{algebraic de Rham cohomology} of $Y$ at $p$. By construction, $\mathcal{H}^i_p(Y)$ is the local hypercohomology $\HH^i_p(\mathfrak{X},\Omega_{\mathfrak{X}}^{\bullet})$  where $\Omega_{\mathfrak{X}}^{\bullet}$ is the completion of  de Rham complex $\Omega_{A/k}^{\bullet}$ along with $Y$. For more details, see \cite[\S III.1]{H}.
If $y\in Y$ is  not  a closed  point, we may look at a representative  of the residue field $\zeta:k(y)\to \widehat{\mathcal{O}}_{Y,y}$. By Cohen's structure theorem
$\widehat{\mathcal{O}}_{Y,y}$ is homomorphic image of a complete regular local ring $ A $. Then we  compute $\mathcal{H}^i_y(Y):=\mathcal{H}^i_{\zeta} (\Spec(\widehat{\mathcal{O}}_{Y,y}))$ similar as  part i). For more details, see \cite[\S III.6]{H}.
Suppose $Y$ is of finite type over $k$ and it is a closed subscheme of an smooth scheme $X$. In a similar vein one can construct algebraic de Rham cohomology  of $Y$ at  a point $p\in Y$.
\end{discussion}

The following may simplify some things from \cite[Theorem 3.1]{lsw}.

\begin{proposition}\label{clsw}
Let $R := \mathbb{C}[x_1,\ldots,x_n]$, $I\lhd_h R$, and $\fm$ be the irrelevant ideal. If $k_0$ is a positive integer such that  $\Supp \HH^k_I(R)\subset \{p\}$  for each $k \geq k_0$. Then,  $\HH^k_I(R) = \HH^k_{\fm}(R)^{\mu_k}$ for each $k\geq k_0$, where $\mu_k:=\dim_{\mathbb{C}}\HH^{k}_{\sing,p} (\V(I)_h,\mathbb{C})$.
\end{proposition}

\begin{proof}
By completion we mean completion with respect to $\fm$-adic topology. Remark that $\widehat{R}$ is local.  It is easy to see
$\Hom_R(\HH^{j}_{I}(R),\E_R(\mathbb{C}))\simeq\Hom_{\widehat{R}}(\HH^{j}_{I}(\widehat{R}), \E_{\widehat{R}}(\mathbb{C}))
\quad(\ast)$.
We set $Y:=\V(I)$ and we look at it as a closed subscheme of the smooth scheme $X:=\Spec(R)$. Recall that $Y$ is of finite type over  $k$.
By the excision theorem (see \cite[\S III.3]{H}) we have $\mathcal{H}^j_q(Y)\simeq\mathcal{H}^j_q(\Spec(\widehat{\mathcal{O}}_{Y,q}))\quad(\dag)$.
Recall from \cite[Theorem 2.3]{O}:
\begin{enumerate}
\item[Fact] i):  Let $ A := k[[t_1,\ldots,t_n]]$ and let $\textbf{Y}:=\V(J)$ be a closed subset of $\textbf{X} := \Spec (A)$ and $p$ be the closed point.
 Assume $ s$ is an integer such that $\Supp \HH^i_J(A)\subset \{p\}$ for all for all $i > n - s$. Then there are natural maps:
$A \otimes_ k \mathcal{H}^j_p(\textbf{Y})\to\HH^j_p(\mathfrak{X},\mathcal{O}_{\mathfrak{X}})$ which are isomorphisms for $j <s$ and injective for $j =s$.
\end{enumerate}Set $s:=n-k_0+1$ and $\lambda_j:=\dim_{\mathbb{C}}\HH^j_p(\mathfrak{X},\mathcal{O}_{\mathfrak{X}})$.
We use Fact i) to see$$
\widehat{R} \otimes_\mathbb{C} \mathcal{H}^{j}_q(Y)
\stackrel{(\dag)}\simeq\widehat{R} \otimes_\mathbb{C} \mathcal{H}^{j}_q(\Spec(\widehat{\mathcal{O}}_{Y,q}))\stackrel{i)}\simeq\HH^{j}_q(\mathfrak{X},\mathcal{O}_{\mathfrak{X}})
\stackrel{\ref{fd}}\simeq\D(\HH^{n-j}_{\widehat{I}}(\widehat{R}))
\stackrel{(\ast)}\simeq\D(\HH^{n-j}_{I}(R)).$$Recall that $\mathcal{H}^{n-j}_p(Y)$ is a finite-dimensional $\mathbb{C}$-vector space.
Conclude that  $\HH^{n-j}_{\widehat{I}}(\widehat{R})$ is artinian. This implies that $\HH^{n-j}_{I}(R)$ is artinian  as an $R$-module (and $\widehat{R}$-module).
%Since the $R$-module structure is compatible with the $\widehat{R}$-module structure,
Also, $\HH^{n-j}_{I}(R)\simeq\HH^{n-j}_{I}(R)\otimes_R\widehat{R}\simeq\HH^{n-j}_{\widehat{I}}(\widehat{R})$.
Take another duality to see
$\bigoplus_{\lambda_j}\HH^n_{\fm}(R)\simeq \bigoplus_{\lambda_j}\E_R(R/ \fm)\simeq\D(\widehat{R} \otimes_\mathbb{C} \mathcal{H}^j_p(Y))\simeq\D(\D(\HH^{n-j}_{\widehat{I}}(\widehat{R})))\simeq\HH^{n-j}_{I}(R).$ This in turns
equivalent with $ \bigoplus_{\lambda_{n-k}}\HH^n_{\fm}(R)\simeq\HH^{k}_{I}(R)$ for all $k\geq k_0$.  By comparison theorem (see \cite[IV.3.1]{H}), $\mathcal{H}^{k}_p(Y)\simeq\HH^{k}_{ \sing,p} (Y_h,\mathbb{C})$.
Therefore, $\bigoplus_{\mu_k}\HH^n_{\fm}(R) \simeq\HH^{k}_{I}(A)$ for all $k\geq k_0$.
\end{proof}

Let $ A := k[[t_1,\ldots,t_n]]$ and let $\textbf{Y}:=\V(J)$ be a closed subset of $\textbf{X} := \Spec (A)$ and $p$ be the closed point.
 Ogus proved that de Rham depth of $Y$ is $n-\cd(J)$. Definition of de Rham depth involved in general (not necessarily closed) points.
 We ask:
When is $\mathcal{H}^{n-\cd(J)}_p(\textbf{Y})\neq 0$?
In the situation for which the answer is positive, one can show   that $\Supp(\D(\HH^{\cd(J)}_{J}(A)))=\Spec(A).$
For example,
there are
 situations for which $\HH^{\cd(J)}_{J}(A) $ is artinian (resp. is not artinian).
 In the case $\HH^{\cd(J)}_{J}(A) $ is artinian it follows that
 $\HH^{\cd(J)}_{J}(A)$ is injective and so  $\D(\HH^{\cd(J)}_J(A)) $
 is flat. This property is stronger than $\Supp(\D(\HH^{\cd(J)}_{J}(A)))=\Spec(A).$
\begin{example}

 i)  Let
$R:=k[[x_1,\ldots,x_4]]$ and $\fa=(x_1, x_2)\cap(x_3, x_4)$.
 Hartshorne proved that $\HH^{3}_{\fa}(R)\simeq\HH^{4}_{\fm}(R)$ (see \cite[Example 3]{Hcd}). In view of
 Fact  i) in Proposition \ref{clsw} we have  $$R \otimes_ k \mathcal{H}^1_p(\V(\fa))\stackrel{\simeq}\lo\HH^1_p(\mathfrak{X},\mathcal{O}_{\mathfrak{X}})=\D(\HH^3_{\fa}(R))=\D(\HH^{4}_{\fm}(R))=\D(\E_R(k))=R.$$
 From this we get that $\mathcal{H}^1_p(\V(\fa))=k$. In particular, $\D(\HH^3_{\fa}(R))$ is free and finitely generated.

 ii)  Let
$R:=k[[x_1,x_2, x_3]]$ and $\fa=(x_1)\cap(x_2, x_3)$.
Recall from  \cite[Example 2.9]{Hel1} that $\HH^2_{\fa}(R)=\E_R(\frac{R}{(x_2,x_3)})$.
In particular, $\D(\HH^2_{\fa}(R))$ is flat. If it were be finitely generated then we should
had $\D(\D(\HH^2_{\fa}(R)))$ is artinian. On the other hand, $\E_R(\frac{R}{(x_2,x_3)})\simeq\HH^2_{\fa}(R)\hookrightarrow\D(\D(\HH^2_{\fa}(R))).$ So
$\E_R(\frac{R}{(x_2,x_3)})$ is artinian. This  is a  contradiction. In particular,  $\D(\HH^2_{\fa}(R))$ is flat but not finitely generated.

 iii)  Let
$R:=k[[x_1,x_2]]$ and $\fa=(x_1)$.
It follows from the definition that $\id(\HH^1_{\fa}(R))=1$ and so flat
dimension of $\D(\HH^1_{\fa}(R))$ is one. In particular,  $\D(\HH^1_{\fa}(R))$ is not flat, but
its support is the prime spectrum.
\end{example}

\begin{corollary}\label{lci}Let $R:=\mathbb{Q}[x_1,\ldots,x_n]_{(x_1,\ldots,x_n)}$  and $\fp\in \Spec(R)$ be locally CI.
 Then $\Supp(\D(\HH^i_{\fp}(R)))=\Spec(R)$ provided  $\HH^i_{\fp}(R)\neq 0$.
\end{corollary}

\begin{proof} First we may assume that
$i>\grade(\fp)$.
Let $\fq\in \V(\fp)\setminus\{\fm\}$ and let $j\geq i$. Since $\mu(\fp_{\fq})< i$ we see
$\HH^j_{\fp}(R)_\fq\simeq\HH^j_{\fp_{\fq}}(R_{\fq}) =0$. Trivially, $\HH^j_{\fp}(R)_\fq =0$ when $\fq\nsubseteq \fp$. Thus,  $\Supp \HH^j_{\fp}(R)\subset \{p\}$  for each $j \geq i$.
We deduce from Fact \ref{clsw}.i)  that $ \HH^i_\fp(R) $ is injective (also, this follows from  Lyubeznik's inequality). Recall  that  $ \HH^i_\fp(R) \neq0$.  In
conjunction with Matlis duality, $\D(\HH^i_\fp(R))$ is nonzero and flat.
Consequently, $\Supp(\D(\HH^i_\fp(R)))=\Spec(R)$.
The claim in the case $i=\grade(\fp)$ is subject of  Proposition \ref{simon}.
\end{proof}

\begin{lemma}\label{6gcm}Let $A$ be a regular local ring  essentially of finite type over a field and $J\lhd A$. Suppose $\dim A < 7$.
Then   $\HH^i_J(A) $ is artinian for all $i> \Ht(J)$ provided $A/J$  is  generalized Cohen-Macaulay.
\end{lemma}

\begin{proof}
Indeed, we assume $\dim A=6$. By the mentioned result of Ogus (resp. Hartshorne-Speiser) in the zero characteristic (resp. in the prime characteristic) case we need to show
$\Supp(\HH^i_ J(A))\subseteq \{\fm\}$ for all $i> \Ht(J)$.
In this regard, first we deal with the case $\Ht(J)=1$. Since $A/J$ is locally Cohen-Macaulay,  $J_{\fp}$ is unmixed for all $\fp\in \Spec(A)\setminus\{\fm\}$. By this,
$J_{\fp}$ is  principal for all $\fp\in \Spec(A)\setminus\{\fm\}$. Thus $\HH^i_J(A)_{\fp}=0$  for all $\fp\in \Spec(A)\setminus\{\fm\}$ and for all $i>1$.  Consequently, $\Supp(\HH^i_ J(A))\subseteq \{\fm\}$ for all $i> 1$. If $\Ht(J)=2$, then  $\frac{A_{\fp}}{J_{\fp}}$ is $\Se_3$ for any height five prime ideal $\fp\in \V(J)$. Combine this along with \cite[Theorem 3.8(2)]{dt} to see $\cd(J_{\fp})\leq5-3$.
Thus $\HH^i_ J(A)_{\fp}=0$  for all $\fp\in \Spec(A)\setminus\{\fm\}$ and for all $i>2$, i.e., $\Supp(\HH^i_J(A))\subseteq \{\fm\}$ for all $i> 2$.  The next case is $\Ht(J)=3$.
Then  $\frac{A_{\fp}}{J_{\fp}}$ is $\Se_2$ for any height five prime ideal $\fp\in \V(J)$.   Recall from \cite[Proposition 3.1]{Ma}:\begin{enumerate}
\item[Fact]A) Let $B$ be a regular local ring containing a field and $\fa\lhd B$.
If $\depth(B/\fa) \geq 2$, then $\cd(\fa) \leq \dim B-2$.
\end{enumerate}In the light of this fact $\cd(J_{\fp})\leq5-2$. It turns out that $\Supp(\HH^i_J(A))\subseteq \{\fm\}$ for all $i> 3$.
Suppose now that $\Ht(J)=4$. Due to $(HLVT)$ we know $\HH^5_J(A)_{\fp}=0$ for any height five prime ideal $\fp$.
 This indicates that $\Supp(\HH^i_J(A))\subseteq \{\fm\}$ for all $i> 4$.
Eventually, we assume  $\Ht(J)\geq5$. Let $i> \Ht(J)$. By $(GVT)$, $\HH^{>6}_J(A)=0$. We establish the desired claim if we recall that $\HH^{6}_J(A) $  is artinian.
\end{proof}

\begin{corollary}\label{6}
Let $R:=\mathbb{Q}[x_1,\ldots,x_6]_{(x_1,\ldots,x_6)}$  and $\fp\in \Spec(R)$ be generalized Cohen-Macaulay.
 Then $\Supp(\D(\HH^i_{\fp}(R)))=\Spec(R)$ provided  $\HH^i_{\fp}(R)\neq 0$.
\end{corollary}

\begin{proof}
In the light of Lemma \ref{6gcm} we see
$\Supp(\HH^i_ {\fp}(R))\subseteq \{\fm\}$ for all $i\neq \Ht(\fp)$.
We use   Fact \ref{clsw}.i)  to see $\id_R(\HH^i_{\fp}(R))= \dim_R(\HH^i_{\fp}(R))$ for all $i\neq \Ht(\fp)$.
Thus, $\D(\HH^i_{\fp}(R))$ is flat provided $\HH^{i}_I(R)\neq0$ and that $i\neq \Ht(\fp)$. This shows   that $\Supp(\D(\HH^i_{\fp}(R)))=\Spec(R)$ provided  $\HH^i_{\fp}(R)\neq 0$ and
that $i\neq\Ht(\fp)$.
The claim in the case $i=\grade(\fp)=\Ht(\fp)$ is subject of  Proposition \ref{simon}. The proof is now complete.
\end{proof}

\section{   An application: computing the injective dimension}
Let $A$ be  a noetherian
ring  of   finite Krull dimension and $\fa\lhd A$. Recall from \cite{sga2} that $\f_{\fa}(M)$  is the largest integer
$n$   such that $\HH^i_{\fa}(M)$ is a finitely generated
for all $i < n$.
Recall from \cite{Hcd} that
$\q_{\fa}(A)$ is the greatest  $i$ such that $\HH^{i}_{\fa}(A)$
is not artinian.
 In the graded setting, the artinian property of $\HH^{i+1}_{\fa}(A)$ implies that the  vector spaces
$\HH^i(\Proj(A/ \fa),\mathcal{O}(n))$ are finite for all $n\in \mathbb{Z}$.
Suppose $U$ is a   scheme of finite Krull dimension and finite type over $k$.
Define the integer $\p(U)$ to be the largest integer
$n$   such that $\HH^i (U,\mathcal{F})$ is a finite-dimensional $k$-vector space
for all $i < n$, and for all locally free sheaves $\mathcal{F}$ on $U$. One can define $\q(U)$ in the similar vein, see \cite[Page 91]{H4}.
Let $Y\subset\PP^n_k$ be a non-empty closed, and
$U := \PP^n_k\setminus Y$. By \cite[Ex. III.5.8]{H4}, $\p(U) \leq \q(U)$. Its local version is:

\begin{fact}(See \cite{ADT}) \label{fir}
 Let $A$ be as above. If
$q_{\fa}(A)>0$, then  $\f_{\fa}(A)\leq \q_{\fa}(A)$.
\end{fact}
 Here, $R$ is  regular and contains a field.
In the previous sections we make use of the fundamental inequality $\id_R(\HH^i_ I(R))\leq \dim_R(\HH^i_ I(R))$.
Over certain polynomial rings this is in fact an equality, see \cite{p}.
There is a difference between the local and global case, because of examples due to Hellus even the ring is 3-dimensional.

\begin{corollary}\label{lci}
Let $R$ be a regular local ring containing  a field and $I$   locally CI. If $I$ is equi-dimensional, then \begin{equation*}
\id_R(\HH^i_ I(R))= \dim_R(\HH^i_ I(R))=\left\{
\begin{array}{rl}
\dim (\frac{R}{I}) & \  \   \   \   \   \ \  \   \   \   \   \ \text{if } i=\Ht(I)\\
0 &\  \   \   \   \   \ \  \   \   \   \   \ \text{if } i\neq\Ht(I) and \HH^i_ I(R)\neq 0,\\
-\infty & \  \   \   \   \   \ \  \   \   \   \   \ \text{otherwise }
\end{array} \right.
\end{equation*}
In particular, $\q_I(R)=\Ht(I)=\f_I(R)$ provided $\rad(I)\neq\fm$.\end{corollary}

In fact we need a weaker assumption:  $I_{\fp}$ up to radical is generated
by any sequence of length equal to height of $I_{\fp}$ for all $\fp\in \V(I)\setminus\{\fm\}$.

\begin{proof}
First we prove the desired property for all
$i>\Ht(I)$. We know from the equi-dimensional  assumption that $\Ht(I_{\fp})=\Ht(I)$ for all $\fp\in \V(I)\setminus\{\fm\}$.
Let $\fq\in \V(I)\setminus\{\fm\}$ and let $j\geq i$. Since $\mu(I_{\fq})< i$ we see
$\HH^j_{I}(R)_\fq\simeq\HH^j_{I_{\fq}}(R_{\fq}) =0$. Trivially, $\HH^j_{I}(R)_\fq =0$ when $\fq\nsubseteq I$. Thus,  $\Supp \HH^j_{I}(R)\subset \{\fm\}$  for each $j \geq i$.
By  Fact i) in Proposition \ref{clsw},  $\HH^i_{I}(R)\simeq\E_R(R/ \fm)^t$  for some $t\in \mathbb{N}_0$  (the prime characteristic case is in \cite[Theorem 2.3]{HS}).
Without loss of the generality we assume $t\neq0 $. This means that  $\id_R(\HH^i_ I(R)) = \dim_R(\HH^i_ I(R))=0$.  In the remaining case
$i:=\Ht(I)$ we look at \cite[Proposition 3.5]{D}:
 \begin{enumerate}
\item[Fact] A) Let $A$ be a regular local ring
containing a field and $I$ be of   height $h$. Then
$\id(\HH^h_{I}(A)) = \dim(\HH^h_{I}(A))$.
\end{enumerate}
This paragraph works for any rings: Let $\fp\in\V(I)$ be of height $h:=\Ht(I)$. By Grothendieck's
non-vanishing theorem,  $\HH^h_{I}(R)_{\fp}\simeq\HH^h_{I_{\fp}}(R_{\fp})\simeq\HH^h_{\fp R_{\fp}}(R_{\fp})\neq 0$. Since
$\Supp(\HH^h_{I}(R))\subset\V(I)$,  we deduce that $\dim_R(\HH^h_ I(R))=\dim (\frac{R}{I})$.

To see the particular case we assume that $\rad(I)\neq\fm$. By the first part
 $\HH^{>h}_ I(R)$ is artinian.
Since $\dim(\HH^{\Ht(I)}_{I}(R))=\dim (R/I)>0$, $\HH^{\Ht(I)}_{I}(R)$ is not artinian. We conclude that $\q_I(R)=\Ht(I)$.
Then $\Ht(I)\leq\f_I(R)\stackrel{\ref{fir}}\leq\q_I(R)=\Ht(I)\footnote{Also, this follows by an observation of Hellus: He used Lyubeznik's inequlity to show if $R$ is a regular local ring containing a field then
$\HH^i_{I}(R)$ is finitely generated only if it vanishes. From this, $\Ht(I)=\f_I(R)$.  We  remark that Conjecture 1.1 implies Hellus' observation even in the mixed characteristic case.}$ and so $\q_I(R)=\Ht(I)=\f_I(R)$.
\end{proof}

\begin{example}\label{im}
Concerning the equality  $\id_R(\HH^i_ I(R))= \dim_R(\HH^i_ I(R))$ all of the assumptions are important.

i) The equi-dimensional assumption is important. Let $R:=K[[x,y,z]]$ and $I=(xy,xz)$. Let $\fp\in\V(I)\setminus\{\fm\}$.
Suppose first that $\fp\supset(y,z)$. Since $\fp\neq\fm$, $x\notin\fp$. As $x$ is invertible in $R_{\fp}$, $I_{\fp}=(y,z)_{\fp}$. This is complete-intersection.
 We may assume that $\fp\supset(x)$. If $\fp=(x)$, then both of $y$ and $z$ are invertible in $R_{\fp}$, and so $I_{\fp}=(x)_{\fp}$. This is complete-intersection.
If $\fp\supsetneqq(x)$. Then at least one of $\{y,z\}$ is not in $\fp$, because $\fp\neq\fm$. By symmetry, we may assume that $y\not\in\fp$. Since $y$ is invertible in $R_{\fp}$, $I_{\fp}=(x,xz)_{\fp}=(x)_{\fp}$. This is complete-intersection.  In conclusion, we have proved that $I$ is locally CI. Clearly, $I=(x)\cap(y,z)$ is not equi-dimensional. It is shown
in \cite[Example 2.9]{Hel3} that  $\id_R(\HH^2_ I(R))=0<1= \dim_R(\HH^2_ I(R))$. Also, we remark that  $\q_I(R)=2>1=\Ht(I)$.

ii) The locally CI assumption is important. Let $R:=K[[x_1,\ldots,x_7]]$ and $I=(x_1,x_2,x_3)\cap(x_4,x_5,x_6)\cap(x_7,x_1,x_2)$. Trivially
$I$ is equi-dimensional. We left to the reader to check that $I$ is not locally complete-intersection (this follows by the above corollary).
By \cite[Example 2.11]{Hel3} $\id_R(\HH^5_ I(R))=0<1= \dim_R(\HH^5_ I(R))$. Also, we remark that  $\q_I(R)=5>3=\Ht(I)$.

iii) The regular  assumption is important: We look at the complete-intersection ring $A =
\frac{\mathbb{C}[x, y, u, v]_{\fn}}{(xy - ux^2 - vy^2)}$, where  $\fn = (x, y, u, v)$. Set $J := (y, u, v)$.
This is a height-two prime ideal. Clearly,
$J$ is Cohen-Macaulay. We set $R:=\widehat{A}$ and look at $I:=\widehat{J}$. Clearly,
$R/I$ is Cohen-Macaulay and 1-dimensional. This means that
$I$ is equi-dimensional and $\V(I)\setminus\{\fm\}=\min(I)$ where $\fm:=\widehat{\fn A}$ stands for the maximal ideal. Thus,
for any $\fp\in \V(I)\setminus\{\fm\}$ we see  that $I_{\fp}=\rad(a,b)$ where $a$ and $b$
is any parameter sequence. Hence $I_{\fp}$ up to radical is generated
by a  sequence of length equal to height of $I_{\fp}$ for all $\fp\in \V(I)\setminus\{\fm\}$. We know from $(HLVT)$ that $H:=\HH^3_{I}(R)\neq 0$. This module is artinian and has a nonzero annihilator.
Suppose on the contradiction that $\id_R(H) = \dim_R(H)$. By Matlis theory,
$H=\bigoplus \E_R(R/ \fm)^t$.
It turns out that $H$ is faithful, a contradiction.
\end{example}

\begin{corollary}\label{d-2}
Let $R$ be a regular local ring containing  a field, $I$ be equi-dimensional and $\dim(R/I) = 2$.
Then $\id_R(\HH^i_ I(R)) = \dim_R(\HH^i_ I(R))$. In particular, $\q_I(R)=\Ht(I)=\f_I(R)$ provided $\rad(I)\neq\fm$.
\end{corollary}

\begin{proof} Set $d:=\dim R$.
If $\cd(I)=d-2$, then $I$ is cohomologically CI. The claim for such a class of ideals is in \cite[Remark 2.7]{Hel3}.
By $(HLVT)$, we have $d-2=\grade(I)\leq\cd(I)\leq d-1$. Without loss of the generality we may assume that $\cd(I)=d-1$. There are only two spots
for which the local cohomology is nonzero.
It is shown in Lemma \ref{CD1} that $\dim(\HH^{d-1}_I(R))=\id(\HH^{d-1}_I(R))=0$. It remains
to deal
with $\HH^{d-2}_I(R)$. This is subject of Fact \ref{lci}.A).
The particular case is similar to Corollary \ref{lci}.
\end{proof}

\begin{lemma}\label{01n}
Let $R$ be a regular local ring of dimension $n$  containing a field and $R/I$ be Cohen-Macaulay.  Suppose $\Ht(I)\in\{0,1,n-2,n-1,n\}$.
Then $I$ is  cohomologically CI.
\end{lemma}

\begin{proof}
The claim is clear if height of $I$ is zero or $n$.
Suppose $\Ht(I)=1$. Since  $I$ is unmixed and $R$ is $UFD$,
 $I$ is principal. This implies that $ I$ is cohomologically CI. The case $\Ht(I)=n-1$ follows by $(HLVT)$. Finally, we deal with the case $\Ht(I)=n-2$.
We use Fact A) in Lemma \ref{6gcm} to see $n-2=\grade(I,R)\leq\cd(I) \leq n-2$. Clearly, $I$ is  cohomologically CI.
\end{proof}

\begin{remark}\label{b5}
Let $R_d:=\mathbb{Q}[x_1,\ldots,x_d]$ and  $I\lhd_h R$ be  Cohen-Macaulay. Suppose  either $d<6$ or $I$ is monomial.
 Then $I$ is cohomologically CI.
\end{remark}

\begin{proof}
First, we prove the monomial case (in Remark \ref{rich} we will drive this from  Richardson's idea). We adopt the notation presented in Proposition \ref{mo}. A minimal free-resolution $ \mathcal{F}$ of $R/I$ is a complex consisting of $\mathbb{Z}^d$-graded modules equipped with
$\mathbb{Z}^d$-graded differentials. Then  $\F_{q^n}(\mathcal{F})$ is  a minimal free resolution of $R/I^{[q^n]}$. Due to the graded version of Auslander-Buchsbaum formula
 $$\pd_R(\frac{R}{I^{[q^n]}}) =\pd(\frac{R}{I}) = d-\depth(\frac{R}{I}) = d-\dim (\frac{R}{I}) = \Ht(I)=\grade(I,R)=:\ell.$$ Thus, $\HH^i_I(R)={\varinjlim}_n\Ext^{i}_{R}(R/I^{[q^n]}, R)=0$ for any $i\neq \ell$.
 Now we deal with any homogeneous ideals over $R_d$ with  $d<6$. We present the proof when $d=5$.
Suppose $\Ht(I)=2$. Then $R/I$ is $\Se_3$. Recall from \cite[Theorem 3.5]{Ma} that $\cd(I) = 5-3$, i.e., $ I$ is cohomologically CI.
The remaining cases are  in Lemma \ref{01n}.
\end{proof}

 Let $R:=\mathbb{Q}[x_{ij}:1\leq i\leq2, 1\leq j\leq3]$ and $I_2$  be the 2-minors of $(x_{ij})$. By  \cite{HS}, $I_2$ is a graded Cohen-Macaulay ideal and $\grade(I_2,R)=2<3=\cd(I_2)$. One can show that  $\id_R(\HH^3_{I_2}(R)) = \dim_R(\HH^3_{I_2}(R))=0$,
and $\id_R(\HH^2_{I_2}(R)) = \dim_R(\HH^2_{I_2}(R))=4$.
An ideal $I$ is called  almost CI, if $\mu(I)\leq\Ht(I)+1$. For example, $I_2$ is almost CI.
We show:

\begin{corollary}\label{g3}
Let $R$ be a regular local ring of dimension $6$ containing  a field $k$ and  $I$ be  Cohen-Macaulay. Suppose either i) $R$ is essentially of finite type over $k$,
or ii) $I$ is almost CI.
Then
$\id (\HH^i_ I(R)) = \dim (\HH^i_ I(R))$ for any $i$.
\end{corollary}

In  Corollary \ref{lgcm} we will extend this by assuming $I$ is generalized Cohen-Macaulay.

\begin{proof}The claim is clear for cohomologically CI, see e.g. \cite[Remark 2.7]{Hel3}. We investigated the cases $\Ht(I)\in\{0,1,4,5,6\}$  in  Lemma \ref{01n}. Suppose $\Ht(I)=3$. We prove this without any use of i) and ii). Since $R/I$ is Cohen-Macaulay, we have $\depth(R/I)=3$. In view of  Fact A) in Lemma \ref{6gcm}  we know $\cd(I) \leq 6-2$.  Due to  \ref{lci}.A)  $\id_R(\HH^3_ I(R)) = \dim_R(\HH^3_ I(R))=3$.
Suppose $\HH^4_ I(R)\neq0$ (conjecturally, this never happens).
\begin{enumerate}\item[Fact A)]  (See \cite[Proposition 3.2]{Ma}) Let $A$ be an $n$-dimensional regular local ring containing a field and $\fa$ an ideal.
If $\depth(R/\fa) = k$, then $\dim(\Supp(\HH^{n-i}_{\fa}(A))\leq i-2$ for all $0 \leq i < k$.
\end{enumerate}
Then, $\dim(\HH^{4}_{I}(R))=0$.
By a celebrated  result of Lyubeznik $0\leq\id(\HH^4_{I}(R))\leq\dim(\HH^{4}_{I}(A))=0.$
Without loss of the generality we may   assume  that $\Ht(I)=2$.   In this case $\cd(I)\leq 6-2$. By Fact A) in Corollary \ref{lci} $\id_R(\HH^2_ I(R)) = \dim_R(\HH^2_ I(R))=4$. By the above fact,
 $\id(\HH^4_{I}(R))=\dim(\HH^{4}_{I}(A))$. The proof finishes if we show $\id(\HH^3_{I}(R))=\dim(\HH^{3}_{I}(A))$.
 We have nothing to prove if $\HH^3_ I(R)= 0$.
Suppose it is nonzero, and suppose in the contradiction that
$\dim(\HH^{3}_{I}(R))\geq1$. Let $\fq\in \Supp(\HH^{3}_{I}(R))$ be a prime ideal of height five.
Let $A:=R_{\fq}$ and $J:=IA$.
\begin{enumerate}
\item[Claim)] Let $A$  be a five-dimensional regular local ring  and  $J$ be a Cohen-Macaulay ideal of height two. Suppose either
$\mu(J)\leq 3$ or $A$ is essentially of finite type over a field.  Then
$\HH^{3}_{J}(A)=0$.\end{enumerate} Indeed,  suppose first that $\mu(J)\leq 3$. Due to Hilbert-Burch, there is a $3\times 2$ matrix $\varphi$ and an element $a$ such that
$J=aI_2(\varphi)$. Since $J$ is equi-dimensional we see that $a$ is a unit. Therefore, we can assume that
$J=I_2(\varphi)$.
By \cite[Theorem 1.1]{lsw} one has $\HH^{3}_{J}(A)=\HH^{3}_{I_2(\varphi)}(A)\simeq \HH^{6}_{I_1(\varphi)}(A)=0$, because of $(GVT)$.
The desired claim in the later case is in \cite[Theorem 3.8(2)]{dt}.
In view of this claim, $\HH^{3}_{I}(R)_{\fq}\simeq\HH^{3}_{J}(A)=0$.  This contradiction shows that $\dim(\HH^{3}_{I}(R))=0$.
Consequently, $0\leq\id(\HH^3_{I}(R))\leq\dim(\HH^{3}_{I}(R))=0.$ The proof is now complete.
\end{proof}

\begin{corollary}
Let $R$ be a regular local ring of dimension $5$ containing  a field and  $I\lhd R$ be Cohen-Macaulay. Then
$\id_R(\HH^i_ I(R)) = \dim_R(\HH^i_ I(R))$ for any $i$.
\end{corollary}

\begin{proof}
The only nontrivial case is $\Ht(I)=2$ and $i=3$.
We assume $\HH^3_ I(R)\neq0$ (conjecturally, this never happens). It follows that
 $\dim(\HH^3_ I(R))=0$. In conjunction with Lyubzenik's  inequality, $\id_R(\HH^3_ I(R)) = \dim_R(\HH^3_ I(R))=0$.
\end{proof}

\begin{corollary}\label{g7}
Let $R$ be a regular local ring  essentially of finite type over a field and of dimension $7$. If $I\lhd R$ is Gorenstein,
then
$\id_R(\HH^i_ I(R)) = \dim_R(\HH^i_ I(R))$ for any $i$.
\end{corollary}

\begin{proof}By the same reason as of  Corollary \ref{g3}, it follows from the Cohen-Macaulay property that $\id_R(\HH^i_ I(R)) = \dim_R(\HH^i_ I(R))$
provided $\Ht(\fa)\neq2$. We left its straightforward  modifications to the reader.  The remaining case is $\Ht(\fa)=2$. Here, we use the   Gorenstein property. By a famous
result of Serre (see \cite{Se}),  $I$  is CI. In particular, $\id_R(\HH^i_ I(R)) = \dim_R(\HH^i_ I(R))$ for any $i$.
\end{proof}

Let $J$ be a square-free monomial ideal in a polynomial ring $R:=K[x_1,\ldots,x_n]$ where $K$ is a field   of any characteristic.
Richardson in his thesis proved:
i) $ J$  is Cohen-Macaulay $\Longleftrightarrow J$  is cohomologically CI,
ii) $\HH^{n-j}_{\fm}(R/J)=0\Longleftrightarrow\HH^j_J(R)=0$,
iii)   $\ell(\HH^{n-j}_{\fm}(R/J))<\infty\Longleftrightarrow\HH^j_J(R)$ is artinian,
iv)   If $R/J$ satisfies $\Se_r$  over the punctured homogeneous-spectrum, then $\HH^j_J(R)$ is artinian for all $j\geq n-r$.

\begin{remark}\label{rich} Let $I$ be a  monomial ideal of the above $R$. The following assertions hold:
\begin{enumerate}
 \item[i)] If $I$ is Cohen-Macaulay,   then $\HH^j_I(M)=0$ for all $i>\Ht(I)$ and  for any module $M$. Indeed, the Cohen-Macaulay property decent from a monomial ideal to its radical (see \cite{ht}). We apply this along with the claim in the squarefree case to see $\cd(I)=\Ht(I)$. Since  $\cd(I,M)\leq\cd(I)$, $\HH^j_I(M)=0$ for all $i>\Ht(I)$.
 \item[ii)]   If $\HH^{n-j}_{\fm}(R/I) = 0$ then $\HH^j_I(R) = 0$. Indeed, in view of \cite{ht}
$\dim_K\HH^{n-j}_{\fm}(\frac{R}{\sqrt{I}})_{a} =\dim_K\HH^{n-j}_{\fm}(R/I)_a$ where $a:=(a_1,\ldots,a_n)\in \mathbb{Z}^n$ and $a_i\leq 0$.
Also, by Hochster's formula $\dim_K\HH^{n-j}_{\fm}(\frac{R}{\sqrt{I}})_{a}=0$ if at least   some  $a_i>0$. We deduce from this that $ \HH^{n-j}_{\fm}(\frac{R}{\sqrt{I}}) =0$. In view of
the claim in the squarefree case $\HH^j_I(R) = \HH^j_{\sqrt{I}}(R)=0$.
\item[iii)]  If $\HH^{n-j}_{\fm}(R/I) $ is of finite length, then $\HH^j_I(R)$ is artinian. Indeed, since $\dim_K\HH^{n-j}_{\fm}(\frac{R}{\sqrt{I}})_{a}\leq \dim_K\HH^{n-j}_{\fm}(R/I)_a.$
This implies  that $\ell( \HH^{n-j}_{\fm}(\frac{R}{\sqrt{I}}))<\infty$. In the light of
the squarefree case $\HH^j_I(R) = \HH^j_{\sqrt{I}}(R)$ is artinian.
\item[iv)] If $R/I$ satisfies $\Se_r$ condition over the punctured  homogeneous-spectrum, then $\HH^j_I(M)$ is artinian for all $j\geq n-r$  for any finitely generated module $M$. Indeed, in a  similar vein as iii) one can show that   $\HH^j_I(R)$ is artinian for all $j\geq n-r$. The module case follows from the standard reduction, see \cite[Lemma III 4.10]{PS}.
\item[v)] The converse of  i),  ii),  iii)  and  iv) is not  true. We look at  $I: =( a^6,a^5b,ab^5,b^6,a^4b^4c,a^4b^4d,a^4e^2f^3,b^4e^3f^2)\lhd R := K[a,b,c,d,e,f].$
 In order to show $I$  is  cohomologically CI, we revisit
 \cite[Theorem 1.4(iv)]{lim} to see $\cd(I)=\cd(\sqrt{I})=\dim R-\depth(\frac{R}{\sqrt{I}})=6-\depth (\frac{R}{(a,b)})=2.$
We are going to show $\HH^{6-5}_{\fm}(R/I) $ is not of finite length. Suppose on the contradiction that
it is of finite length. Then its Matlis dual $\Ext^5_R(R/I,R)\simeq\Ext^4_R(I,R)$ is of finite length. We use Macaulay 2:\\
i1 : R=QQ[a,b,c,d,e,f]\\i2: $I=ideal(a^ 6,a^5*b,a*b^5,b^6,a^4*b^4*c,a^4*b^4*d,a^4*e^2*f^3,b^4*e^3*f^2)$\\i3: $M=Ext^4(I,R)$\\i4: dim M\\o4 = 1\\
One may use this to see that $\Ext^5_R(R/I,R)$ is not artinian. This contradiction shows that $\HH^{1}_{\fm}(R/I) $ is not of finite length. But,  $\HH^5_ I(R)=0$ is artinian.
So,  converse of the first three items is not true.
In order to fail the reverse of
iv)  we cite \cite[Example 3.3.2]{richar}.
\end{enumerate}
 \end{remark}

Recall   for the squarefree monomial ideal $I$ that $\ell(\HH^{n-j}_{\fm}(R/I))<\infty\stackrel{(\ast)}\Longleftrightarrow\HH^j_I(R)$ is artinian.
This  implies that $\q_{I}(R)=\Ht(I)$ provided $I$   is generalized  Cohen-Macaulay and  $\rad(I)\neq\fm\quad(+)$ . We observed that only half of $(\ast)$ works for monomial ideals.
Despite of this, the following drops the squarefree assumption from $(+)$:

\begin{corollary}\label{q}
Let $I$ be a  monomial ideal in a polynomial ring $R$ with $n$ variables over a field $K$  of any characteristic.
Suppose $I$ is generalized  Cohen-Macaulay. If $\Ht(I)< n$, then $\q_{I}(R)=\Ht(I)=\f_I(R)$.
\end{corollary}

\begin{proof}
 By definition $\HH^{n-j}_{\fm}(R/I) $ is of finite length for all $j\neq \Ht(I)$.
In view of Remark \ref{rich}  $\HH^j_I(R)$ is artinian all $j\neq \Ht(\fa)$. We observed in Corollary \ref{lci} that $\HH^{ \Ht(I)}_{I}(R)$ is not artinian.
We apply this to see $\q_{I}(R)=\Ht(I)=\f_I(R)$.
\end{proof}

\begin{corollary}\label{lgcm}
Let $R$ be a regular local ring  essentially of finite type over a field and $I$  be  generalized Cohen-Macaulay.  Suppose either $\Char R>0$ or $\dim R < 7$.
Then $\id (\HH^i_ I(R))= \dim (\HH^i_ I(R))$ for any $i$. Also, $\q_I(R)=\Ht(I)=\f_I(R)$ provided $\rad(I)\neq\fm$.
\end{corollary}

\begin{proof} In the light of Fact \ref{lci}.A)  we see $\id(\HH_I^{\Ht(I)}(R)) = \dim(\HH^{\Ht(I)}_{I}(R))=\dim (R/I)$.
Without loss of generality we may assume that $i\neq \Ht(I)$.
In the case $\Char R>0$ we know from  \cite[Th\'{e}or\`{e}me III.4.9]{PS} that
$\dim(\HH^{d-j}_{I}(R))\leq0$ for all $j>\dim (R/I)$.
We have $0\leq\id(\HH^{d-j}_{I}(R))\leq\dim(\HH^{d-j}_{I}(R))=0$
provided it is nonzero.  This proves $\id (\HH^i_ I(R))= \dim (\HH^i_ I(R))$ for all $i\neq \Ht(I)$.
Now we deal with the case $\dim R < 7$. In the light of Lemma \ref{6gcm}, one has
$\Supp(\HH^i_ I(R))\subseteq \{\fm\}$ for all $i\neq \Ht(I)$.
We use  Fact \ref{clsw}.i) to see $\id (\HH^i_ I(R))= \dim (\HH^i_ I(R))$ for all $i\neq \Ht(I)$. This completes the proof of first part.
Suppose now that $\rad(I)\neq\fm$. Since $\dim(\HH^{\Ht(I)}_{I}(R))=\dim( R/I)>0$ we get that $\q_I(R)=\Ht(I)=\f_I(R)$.
\end{proof}

\begin{example} (Hartshorne's skew lines in $\PP^3$) Let $R=k [x_1,\ldots,x_4] $. Set $I:=(x_1,x_2)\cap(x_3,x_4)$. Hartshorne used Mayer--Vietoris to show  $\q_I(R)=2$, see \cite[Example 3]{Hcd}. We left to the reader to check that this is an example of:
i) $I$  is locally CI  and equi-dimensional, ii) $I$ is equi-dimensional and $\dim(R/I) = 2$, iii) $I$ is monomial and $R/I$ is generalized Cohen-Macaulay, iv) $\dim R<7$  and $R/I$ is generalized Cohen-Macaulay.
\end{example}
%%%%%%%%%%%%%%%%%%%%%%%%%%%%%%%%%%%%%%%%%


\begin{thebibliography}{99}

\bibitem{ADT}M. Asgharzadeh,
K. Divaani-Aazar and M. Tousi, {\it  The finiteness dimension of local
cohomology modules and its dual notion}, J. Pure Appl. Algebra,  {\bf 213}
(2009),  321--328.


\bibitem{BS}{M. Brodmann and R.Y. Sharp}, {\it
Local cohomology: an algebraic introduction with geometric
applications}, Cambridge Univ. Press, {\bf 60}, Cambridge, 1998.

\bibitem{dt}H.
Dao and S. Takagi, \emph{On the relationship between depth and cohomological dimension}, Compos. Math. {\bf152} (2016),  876-–888.



\bibitem{D}
M. Dorreh, \emph{On the injective dimension of F-finite modules and holonomic D-modules}, Illinois J. Math. {\bf60} (2016),  819–-831.

\bibitem{mac}
D. Grayson,  and M. Stillman,
\emph{Macaulay2, a software system for research in algebraic geometry},
Available at http://www.math.uiuc.edu/Macaulay2/.

\bibitem{sga2}
A. Grothendieck,
\emph{Cohomologie locale des faisceaux coh\'{e}rents et th\'{e}or�mes de Lefschetz locaux et globaux} (SGA 2).
S\'{e}minaire de G\'{e}om\'{e}trie Alg\'{e}brique du Bois Marie  1962. Amsterdam: North Holland Pub. Co. (1968).

\bibitem{HS}
R. Hartshorne and R. Speiser, \emph{Local cohomological dimension in characteristic p}, Ann. of Math. (2) {\bf105} (1977), 45-–79.

\bibitem{H} R.
Hartshorne, \emph{On the de Rham cohomology of algebraic varieties}, IHES Publ. Math. {\bf 45} (1975), 5--99.

\bibitem{Hcd} R.
Hartshorne, \emph{Cohomological dimension of algebraic varieties}, Ann. of Math. (2) {\bf 88} (1968), 403--450.

\bibitem{H4} R.
Hartshorne, \emph{Ample subvarieties of algebraic varieties}, Notes written in collaboration with C. Musili. LNM, {\bf156}, Springer-Verlag, Berlin-New York 1970.

\bibitem{Hel3}
M. Hellus and J. St\"{u}ckrad, \emph{Matlis duals of top local cohomology modules},  Proc. AMS., {\bf136} (2008), 489--
498.

\bibitem{Hel1}
M. Hellus, \emph{
A note on the injective dimension of local cohomology modules},  Proc. AMS., {\bf136} (7)(2008) 2313--2321.


\bibitem{Hel2}{M. Hellus}, {\it Local Cohomology and
Matlis duality}, Habilitationsschrift, Universit\"{a}t Leipzig, (2007),
arXiv:0703124.



\bibitem{ht} J.
Herzog, Y. Takayama, N. Terai, {\it On the radical of a monomial ideal}, Arch. Math. (Basel) {\bf85} (2005), no. 5, 397–-408.


\bibitem{li}G. Lyubeznik and T. Yildirim,
\emph{On the Matlis duals of local cohomology modules},   Proc. AMS,
{\bf  146} (2018), 3715–-3720.

\bibitem{lsw}G. Lyubeznik, A.K. Singh  and U. Walther,
\emph{Local cohomology modules supported at determinantal ideals},
J. Eur. Math. Soc. {\bf18} (2016), no. 11, 2545--2578.





\bibitem{lin} G. Lyubeznik,
\emph{Finiteness properties of local cohomology modules (an application of D-modules
to commutative algebra)}, Invent. Math.  {\bf113}, 41--55, 1993.




\bibitem{lim}G. Lyubeznik,
\emph{On the local cohomology modules $H^i_{\fa}(R)$ for ideals $\fa$ generated by monomials in an R-sequence},  complete-intersections (Acireale, 1983), 214--220, Lecture Notes in Math.,  {\bf 1092}, Springer, Berlin, 1984.

\bibitem{mat} E. Matlis, \emph{The Koszul complex and duality}, Comm. Algebra {\bf1} (1974), 87–-144.

\bibitem{O}{A. Ogus}, {\it Local cohomological dimension of algebraic varieties},
Ann. Math., {\bf 98}, (1973), 327-365.

\bibitem{PS}{C. Peskine and L. Szpiro}, {\it Dimension projective finie
et cohomologie locale}, Publ. Math. IHES., {\bf42}, (1972), 47--119.

\bibitem{p}
T. J. Puthenpurakal, {\it On injective resolution of local cohomology modules}, Illinois J. Math. {\bf58} (2014), 709–-718.


\bibitem{richar}
A. Richardson, {\it Some duality results in homological algebra}, Ph.D. Dissertation,
Univ. of Illinois, 2002.

\bibitem{rob}P.
 Roberts,   \emph{Two applications of dualizing complexes over local rings}, Ann. Scient. \'{E}c. Norm. Sup. $4^e$ Serie, t {\bf9}(1976), 103-–106.



\bibitem{Sch}
P. Schenzel,  {\it On formal  local cohomology and
connectedness}, J. Algebra, {\bf 315}(2), (2007), 894--923.

\bibitem{Se}J.-P.
Serre, {\it Sur les modules projectifs}, Seminaire Dubreil, 1960/61.

\bibitem{Si} A.M.
Simon, {\it Some homological properties of complete modules}, Math. Proc. Cambridge Philos. Soc. {\bf108} (1990),  231-–246.

\bibitem{Ma} M. Varbaro, \emph{Cohomological and projective dimensions}, Compos. Math. {\bf149} (2013),  1203–-1210.

\end{thebibliography}
\end{document}